\documentclass[11pt]{amsart}
\usepackage{amssymb, latexsym}
\theoremstyle{plain}
\newtheorem{theorem}{Theorem}

\newtheorem {lemma}{Lemma}
\newtheorem{proposition}{Proposition}

\newtheorem*{1'}{Theorem 1-Bessel}
\newtheorem*{P2'}{Proposition 2-Bessel}
\newtheorem*{P3'}{Proposition 3-Bessel}
\newtheorem*{P4'}{Proposition 4-Bessel}
\newtheorem*{C1'}{Corollary 1-Bessel}

\newtheorem*{2'}{Theorem 2-Bessel}
\newtheorem*{3'}{Theorem 3-Bessel}

\theoremstyle{remark}

\newtheorem*{Remark 1}{Remark 1}
\newtheorem*{Remark 2}{Remark 2}
\newtheorem*{Remark 3}{Remark 3}
\newtheorem*{Remark 4}{Remark 4}

\numberwithin{equation}{section}
\renewcommand{\baselinestretch}{1.35}

\begin{document}

\title[Longest alternating subsequence in  separable permutations] {Mean and variance of the  longest alternating subsequence in a random separable permutation}

\author{Ross G. Pinsky}


\address{Department of Mathematics\\
Technion---Israel Institute of Technology\\
Haifa, 32000\\ Israel}
\email{ pinsky@technion.ac.il}
\urladdr{https://pinsky.net.technion.ac.il/}

\date{}

\begin{abstract}

A permutation is \it separable \rm if it can be obtained from the singleton permutation   by iterating direct sums and skew sums. Equivalently, it is separable  if and only it avoids the patterns 2413 and 3142. Under the uniform probability on separable permutations of $[n]$, let the random variable $A_n$ denote the length of the longest alternating subsequence.
Also, let $A_n^{+,-}$ denote the length of the longest alternating subsequence that begins with an ascent and ends with a descent, and define $A_n^{-,+}, A_n^{+,+}, A_n^{-,-}$ similarly.
By symmetry, the first two and the last two of these latter four random variables are equi-distributed.
We prove that the expected value of any of these five random variables behaves asymptotically as $(2-\sqrt2)n\approx0.5858\thinspace n$.
We also obtained the more refined estimates
that the expected value of
$A_n^{+,-}$ and of $A_n^{-,+}$ is equal to
$(2-\sqrt2)n-\frac14(3-2\sqrt2)+o(1)$ and that the expected value of
$A_n^{+,+}$ and of $A_n^{-,-}$ is equal to
$(2-\sqrt2)n+\frac34(3-2\sqrt2)+o(1)$. Finally, we  show that
the variance of any of the four random variables $A_n^{\pm,\pm}$   behaves asymptotically as
$\frac{16-11\sqrt2}2n\approx0.2218\thinspace n$.

\end{abstract}

\maketitle
\section{Introduction and Statement of Results}\label{intro}
\renewcommand{\baselinestretch}{1.3}
Let $S_n$ denote the permutations of $[n]:=\{1,\cdots, n\}$. Given $\sigma\in S_k$ and $\tau\in S_l$,
the \it direct sum\rm\ of $\sigma$ and $\tau$ is the permutation in $S_{k+l}$ given by
$$
(\sigma\oplus\tau)(i)=\begin{cases}\sigma(i),\ i=1,\cdots, k;\\ \tau(i-k)+k,\ i=k+1,\cdots k+l,\end{cases}
$$
and the \it skew sum\rm\ $\sigma\ominus\tau$
is the permutation in $S_{k+l}$ given by
$$
(\sigma\ominus\tau)(i)=\begin{cases}\sigma(i)+l ,\ i=1,\cdots, k;\\ \tau(i-k),\ i=k+1,\cdots k+l.\end{cases}
$$
A permutation is \it indecomposable\rm\ if it cannot  be represented as the direct sum of two nonempty permutations and is \it skew  indecomposable \rm  if it cannot  be represented as the skew sum
of two nonempty permutations.
A permutation is \it separable \rm if it can be obtained from the singleton permutation   by iterating direct sums and skew sums.
Equivalently, a permutation is separable if it can be successively decomposed and skew decomposed until all of the indecomposable and skew  indecomposable pieces of the permutation are singletons.
For example, using one-line notation, consider the separable permutation  $\sigma=4352167$. It can be decomposed into $43521\oplus12$. Then 43521 can be skew decomposed into 213$\ominus$21 and 12 can be
decomposed into $1\oplus1$.
Now 213 can be decomposed into 21$\oplus$1 and  21 can be skew decomposed into 1$\ominus$1. Finally, again 21 can be skew decomposed into $1\ominus$1.

It is well-known \cite{BBL}   that a permutation is separable if and only it avoids the patterns 2413 and 3142. For more on pattern avoiding permutations, see for example \cite{B}.
The fact that separable permutations  can be enumerated by a closed form generating function (see section \ref{prelim}) makes them rather tractable to analyze.
The study of general pattern avoiding permutations goes back to Knuth's observation \cite{K} that a permutation is so-called stack sortable if and only if it 231-avoiding.  Similarly, the study of separable permutations goes back to
\cite{AN} where it was shown that these are precisely the permutations which are sortable by so-called pop stacks.
Separable permutations also arise in a variety of other applications, for example in bootstrap percolation \cite{SS}
and in connection to   polynomial interchanges where one studies the possible ways that the relative order of the values of a family of polynomials can be modified when crossing a common zero \cite{G}.

Let $\text{SEP}(n)$ denote the set of separable  permutations in $S_n$, and let $P^{\text{sep}}_n$ and $E^{\text{sep}}_n$ denote respectively the uniform probability measure on  SEP$(n)$ and the expectation with respect to that measure.
In this paper we study the length of the longest alternating subsequence in a random separable  permutation.
 An alternating subsequence of length $k$ in a permutation $\sigma=\sigma_1\cdots\sigma_n\in S_n$ is a subsequence of the form
$\sigma_{i_1}>\sigma_{i_2}<\sigma_{i_3}>\cdots\sigma_{i_k}$
or $\sigma_{i_1}<\sigma_{i_2}>\sigma_{i_3}<\cdots\sigma_{i_k}$,
 where $1\le i_1<\cdots<i_k\le n$.
Let $A_n=A_n(\sigma)$ denote the length of the longest alternating subsequence in a permutation $\sigma\in S_n$. Stanley \cite{S} showed that for a uniformly random permutation in $S_n$,
the expectation of $A_n$ is asymptotic to $\frac23 n$ and the variance is asymptotic to $\frac8{45}n$.
The length of the  longest alternating subsequence in a random permutation avoiding a pattern of length three was studied in \cite{FMW}. It  was proven that
the expectation of $A_n$  in a random $\eta$-avoiding permutation in $S_n$ is asymptotic to $\frac12 n$  and
 the variance
is asymptotic to $\frac14n$,  for any pattern $\eta\in S_3$.

We will prove the following theorem.
\begin{theorem}\label{mean}
\begin{equation}\label{meanasymp}
E^{\text{\rm sep}}_nA_n\sim(2-\sqrt2)n\approx .5858\thinspace n.
\end{equation}
\end{theorem}
\noindent \bf Remark.\rm\  Since  a separable permutation is one that avoids the two patterns 2413 and 3142, and since these patterns are alternating, it is not surprising that the expected length of the
longest alternating subsequence in a random separable permutation  is shorter than that of a uniformly random permutation.  Note
though that it is longer than that of a random permutation avoiding any particular pattern of length three. (See the paragraph preceding Theorem \ref{mean}.)

For the proof of Theorem \ref{mean} and also in order to state a result for  the asymptotic behavior of the variance,
we need to differentiate between four  types of alternating sequences---begin with an ascent and end with an ascent; begin with an ascent and end with a descent;  begin
with a descent and end with an ascent; begin with a descent and end with a descent.
 An alternating subsequence of length $k$ which begins with an ascent and ends with an ascent  in a permutation $\sigma=\sigma_1\cdots\sigma_n\in S_n$ is a subsequence of the form
$\sigma_{i_1}<\sigma_{i_2}>\cdots <\sigma_{i_k}$. The other three types are defined similarly.

Denote the length of the longest alternating subsequence of each of the four types of alternating subsequences by $A_n^{\pm,\pm}$.
Since the difference of any two of the random variables $A_n^{\pm,\pm}$ is bounded by two, Theorem \ref{mean} holds with $A_n$ replaced by any of the four random variables $A_n^{\pm,\pm}$.
From symmetry considerations, it is clear that
\begin{equation}\label{A+-distributions}
A_n^{+,+}\stackrel{\text{dist}}{=}A_n^{-,-};\ \ \  A_n^{+,-}\stackrel{\text{dist}}{=}A_n^{-,+}.
\end{equation}

We will prove the following theorem  which refines Theorem \ref{mean} and gives the asymptotic behavior of the variance.
\begin{theorem}\label{var}
\noindent i.
\begin{equation}\label{+-meanrefine}
\begin{aligned}
&E^{\text{\rm sep}}_nA_n^{+,-}=E^{\text{\rm sep}}_nA_n^{-,+}=
(2-\sqrt2)n-\frac14(3-2\sqrt2)+o(1);
\end{aligned}
\end{equation}
\begin{equation}\label{varasymp}
\begin{aligned}
&\text{Var}_n^{\text{\rm sep}}(A_n^{+,-})=\text{Var}_n^{\text{\rm sep}}(A_n^{-,+})\sim\frac{16-11\sqrt2}2n\approx0.2218\thinspace n.
\end{aligned}
\end{equation}
ii. \begin{equation}\label{--meanrefine}
\begin{aligned}
&E^{\text{\rm sep}}_nA_n^{+,+}=E^{\text{\rm sep}}_nA_n^{-,-}=(2-\sqrt2)n+\frac34(3-2\sqrt2)+o(1);
\end{aligned}
\end{equation}
\begin{equation}\label{varasympagain}
\begin{aligned}
&\text{Var}_n^{\text{\rm sep}}(A_n^{+,+})=\text{Var}_n^{\text{\rm sep}}(A_n^{-,-})\sim\frac{16-11\sqrt2}2n\approx0.2218\thinspace n.
\end{aligned}
\end{equation}
\end{theorem}
\noindent\bf Remark.\rm\ From Theorem \ref{var}, it follows that from the limited perspective of mean and variance,
$A_n^{+,+}$ or $A_n^{-,-}$ behaves as a deterministic translation of $A_n^{+,-}$ or $A_n^{-,+}$ by $3-2\sqrt2$.
We note that the quantity $3-2\sqrt2$ plays a fundamental role in the proofs;  it is one of the roots of the generating function corresponding to the count of separable permutations--see section \ref{prelim}.
\medskip

In section \ref{prelim} we present some preliminary material on separable permutations and on alternating subsequences.
In section \ref{genfuncsmean}, we define two generating functions related to the mean of the length of the longest alternating sequence and evaluate them explicitly.
Using one of these generating functions, we prove  Theorem \ref{mean} in section \ref{meanproof}.
In section \ref{genfuncs2nd}, we
define two  generating functions related to the second moment of the length of the longest  alternating sequence and evaluate them explicitly.
Using all four of the above noted generating functions, we prove Theorem \ref{var} in section \ref{varproof}.
In the appendix we state and prove three propositions,  concerning the asymptotic behavior as $n\to\infty$ of the coefficient of $t^n$ in the power series for the functions
$(t^2-6t+1)^{\frac m2}$, for $m\in\{1,-1,-3\}$. These results are critical for the proofs of Theorems \ref{mean} and \ref{var}. More precisely,
formulas \eqref{snasympappendix} from Proposition \ref{snasympprop} and \eqref{oneoverasymp} from Proposition \ref{oneover} are needed for the proof of Theorem \ref{mean}, while their more precise versions, \eqref{snasymprefine} and \eqref{oneoverasymprefine}, along with formula \eqref{-32asymp} in Proposition \ref{-32}
are needed for the proof of Theorem \ref{var}.

\section{Preliminary material}\label{prelim}
Let $s_n=|\text{SEP}(n)|,\ n\ge1$, denote the number of separable permutations in $S_n$.
Let
$$
s(t)=\sum_{n=1}^\infty s_nt^n
$$
denote the generating function of $\{s_n\}_{n=1}^\infty$.
For a separable permutation, define the length of the first indecomposable block and the length of the first skew indecomposable block respectively by
\begin{equation}\label{Bdef}
\begin{aligned}
&B_1^{+,n}(\sigma)=\min\{j:\sigma([j])=[j]\}, \ \sigma\in \text{SEP}(n);\\
&B_1^{-,n}(\sigma)=\min\{k:\sigma([k])=[n]-[n-k]\}, \sigma\in \text{SEP}(n).
\end{aligned}
\end{equation}
By the definition of separable permutations,
for each $\sigma\in \text{SEP}(n)$, with $n\ge2$, exactly one out of
$B_1^{+,n}(\sigma)$ and $B_1^{-,n}(\sigma)$ is equal to $n$,
and by symmetry,
\begin{equation}\label{sym}
|\{\sigma\in \text{SEP}(n):B_1^{+,n}(\sigma)=n|=|\{\sigma\in \text{SEP}(n):B_1^{-,n}(\sigma)=n|=
\frac12 s_n,\ n\ge2.
\end{equation}
That is, half of the permutations in SEP$(n)$, $n\ge2$,  are indecomposable and half are skew indecomposable.
Partitioning SEP$(n)$ by $\{B_1^{+,n}=j\}_{j=1}^n$ (or alternatively, by $\{B_1^{-,n}=j\}_{j=1}^n$), and
using  the concatenating structure of separable permutations, it  follows that
\begin{equation}\label{rec}
s_n=s_1s_{n-1}+\frac12\sum_{j=2}^{n-1}s_js_{n-j}+\frac12 s_n,\ n\ge2,
\end{equation}
and
\begin{equation}\label{Bprob}
P_n^{\text{sep}}(B_1^{+,n}=j)=P_n^{\text{sep}}(B_1^{-,n}=j)=\begin{cases} \frac{s_1s_{n-1}}{s_n}; j=1;\\ \frac{\frac12s_js_{n-j}}{s_n}, j=2\cdots n-1;\\ \frac12, j=n.\end{cases}
\end{equation}
From \eqref{rec}  it is straightforward to show that
\begin{equation}\label{sgenfunc}
s(t)=\frac12(1-t-\sqrt{t^2-6t+1}\thinspace),\ \text{for}\ |t|<3-2\sqrt{2}.
\end{equation}
(Multiply both sides of \eqref{rec} by $t^n$ and sum over $n$ from 2 to $\infty$, and then solve for $s(t)$.)
Using the above formula for  the generating function,  one can prove  that
\begin{equation*}\label{snasymp}
s_n\sim\frac1{2^\frac34\sqrt{\pi n^3}}(3-2\sqrt2)^{-n+\frac12}.
\end{equation*}
We prove this asymptotic formula as well as a more refined version in Proposition  \ref{snasympprop}. We note that
in \cite[p. 474-475]{FS}, the above formula appears with a mistake---instead of $2^\frac34$, one find there 2.
(Our $s_n$ is equal to their $D_{n-1}$.)   The sequence of integers $\{s_n\}_{n=1}^\infty$ is known as the sequence of big Schr\"oder numbers; see A006318 in the \it On-Line Encyclopedia of Integer Sequences\rm.

As noted before the statement of Theorem \ref{var},
we differentiate between four  types of alternating sequences, depending on whether they begin with an ascent or a descent and whether they end with an ascent or a descent,
and denote the length of the longest one of each type  by $A_n^{\pm,\pm}$.
We note that  a subsequence of length two of the form $\sigma_{i_1}<\sigma_{i_2}$ begins with an ascent and ends with an ascent.
We will derive recursion formulas that allow us to obtain generating functions in explicit forms.
In order to make the formulas work, the following definition will be crucial:
\begin{equation}\label{crucialdef}
\begin{aligned}
&\text{A singleton}\ \sigma_{i_1}\ \text{is considered an alternating sequence both of the type}\ (+,-)\\
& \text{and of the type}\ (-,+);\\
&\text{A singleton}\ \sigma_{i_1}\ \text{is not considered an alternating sequence  of the type}\ (+,+)\\
& \text{or of the type}\ (-,-).\\
\end{aligned}
\end{equation}
In light of \eqref{crucialdef},
note that
\begin{equation}\label{n=1or2}
\begin{aligned}
& A_1^{+,+}=A_1^{-,-}\equiv0;\ \ \
A_1^{+,-}=A_1^{-,+}\equiv1;\\
&A_2^{+,-}=A_2^{-,+}\equiv1\\
&A_2^{+,+}(\sigma)=\begin{cases}2, \ \sigma=12,\\ 0, \ \sigma=21;\end{cases}\ \ \
A_2^{-,-}(\sigma)=\begin{cases}0, \ \sigma=12,\\ 2, \ \sigma=21.\end{cases}
\end{aligned}
\end{equation}
\noindent \bf Remark.\rm\ Note that the definition 
\eqref{crucialdef} ensures that $A_n^{+,+}$ and $A_n^{-,-}$ take on only even values and that $A_n^{+,-}$ and $A_n^{-,+}$ take on only odd values.
\medskip

The following  lemma will play an important role. Let $P^{\text{\rm sep}}_n(\cdot\ |B_1^{+;n}=n)$ \big($P^{\text{\rm sep}}_n(\cdot\ |B_1^{-;n}=n)$\big) denote the distribution
$P^{\text{\rm sep}}_n$ conditioned on $\{B_1^{+;n}=n\}$ \big(\{$B_1^{-;n}=n$\}\big).
\begin{lemma}\label{revcomplemma}
\noindent i. The distribution of $A_n^{+,+}$ \big($A_n^{-,-}$ \big) under    $P^{\text{\rm sep}}_n(\cdot\ |B_1^{+;n}=n)$ coincides with the distribution of $A_n^{-,-}$ \big( $A_n^{+,+}$ \big) under
 $P^{\text{\rm sep}}_n(\cdot\ |B_1^{-;n}=n)$;

 \noindent ii. The distribution of $A_n^{+,-}$ \big($A_n^{-,+}$\big)  under   $P^{\text{\rm sep}}_n(\cdot\ |B_1^{+;n}=n)$  coincides with the distribution of $A_n^{-,+}$
 \big($A_n^{+,-}$\big) under
  $P^{\text{\rm sep}}_n(\cdot\ |B_1^{-;n}=n)$;

\noindent iii. The distributions of $A_n^{+,-}$ and $A_n^{-,+}$ coincide under  $P^{\text{\rm sep}}_n(\cdot\ |B_1^{+;n}=n)$.
\end{lemma}
\begin{proof}
Recall that the \it reverse\rm\  of a permutation $\sigma=\sigma_1\cdots\sigma_n$ is the permutation $\sigma^{\text{rev}}:=\sigma_n\cdots\sigma_1$,
and the \it complement\rm\ of $\sigma$ is the permutation
$\sigma^{\text{com}}$ satisfying
$\sigma^{\text{com}}_i=n+1-\sigma_i,\ i=1,\cdots, n$.
Let $\sigma^{\text{rev-com}}$ denote the permutation obtained by applying reversal and then complementation to $\sigma$ (or equivalently, vice versa).
All of these operations are bijections of $\text{SEP}(n)$; indeed, they all preserve the property of being simultaneously 2413 and 3142 avoiding.
It is easy to see that
$$
B_1^{+,n}(\sigma)=1\ \Leftrightarrow\ B_1^{-,n}(\sigma^{\text{rev}})=n\ \Leftrightarrow\   B_1^{-,n}(\sigma^{\text{com}})=n,\ \sigma\in S_n.
$$
From the above facts, it follows that
\begin{equation}\label{bijections}
\begin{aligned}
&\sigma\to\sigma^{\text{com}}\ \text{is a bijection of}\ \text{SEP}(n)\cap\{B_1^{+,n}=n\}\ \text{to}\ \text{SEP}(n)\cap\{B_1^{-,n}=n\};\\
& \sigma\to\sigma^{\text{rev-com}}\ \text{is a bijection of}\ \text{SEP}(n)\cap\{B_1^{+,n}=n\}\ \text{to itself}.
\end{aligned}
\end{equation}
Also, it is easy to check that
\begin{equation}\label{tocheck}
\begin{aligned}
&A_n^{+,+}(\sigma)=A_n^{-,-}(\sigma^{\text{com}}),\ \sigma\in S_n;\\
&A_n^{+,-}(\sigma)=A_n^{-,+}(\sigma^{\text{com}});\\
&A_n^{+,-}(\sigma)=A_n^{+,-}(\sigma^{\text{rev}}).
\end{aligned}
\end{equation}
From the latter two equations in \eqref{tocheck}, it follows that
\begin{equation}\label{tocheckagain}
A_n^{+,-}(\sigma)=A_n^{-,+}(\sigma^{\text{rev-com}}).
\end{equation}
Part (i) of the lemma follows from the first line of \eqref{bijections} and the first line of \eqref{tocheck}; part (ii) follows from the first line of \eqref{bijections} and the second line of \eqref{tocheck};
and part (iii) follows from the second line of \eqref{bijections} and \eqref{tocheckagain}.
\end{proof}

\section{Generating functions related to the mean of $A_n$}\label{genfuncsmean}

Define
\begin{equation}\label{c'sd's}
\begin{aligned}
&c_n^{\pm,\pm}=E_n^{\text{sep}}(A_n^{\pm,\pm}|B_1^{+,n}=n);
&d_n^{\pm,\pm}=E_n^{\text{sep}}(A_n^{\pm,\pm}|B_1^{-,n}=n).
\end{aligned}
\end{equation}
Define the generating functions
\begin{equation}\label{Ggenmean}
\begin{aligned}
&G^{+,-}(t)=\sum_{n=1}^\infty s_nc_n^{+,-}t^n;\\
&G^{-,-}(t)=\sum_{n=1}^\infty s_nc_n^{-,-}t^n.
\end{aligned}
\end{equation}

We will prove the following theorem.
\begin{theorem}\label{Ggenmeanthm}
\begin{equation}\label{Ggenmeanform}
\begin{aligned}
& G^{+,-}(t)=\frac{t(1-t)}{\sqrt{t^2-6t+1}};\\
&G^{-,-}(t)=\frac{t(1-t)}{\sqrt{t^2-6t+1}}-\frac{t(1-t)}2-\frac t2\sqrt{t^2-6t+1}.
\end{aligned}
\end{equation}
\end{theorem}
\begin{proof}
By \eqref{sym} we have
\begin{equation}\label{exp-cond+cond}
E^{\text{\rm sep}}_nA_n^{\pm,\pm}=\frac12c_n^{\pm,\pm}+\frac12d_n^{\pm,\pm}.
\end{equation}
Conditioning on $\{B_1^{+,n}=j\}, j=1,\cdots, n$,
we obtain the  equation
\begin{equation}\label{exp-condj+-}
\begin{aligned}
&E^{\text{\rm sep}}_nA_n^{\pm,\pm}=\sum_{j=1}^nP^{\text{\rm sep}}_n(B_1^{+,n}=j)E{\text{\rm sep}}_n(A_n^{\pm,\pm}|B_1^{+,n}=j).
\end{aligned}
\end{equation}
Consider  \eqref{exp-condj+-} with $A_n^{+,+}$. From the definition of $A_n^{+,+}$ and the concatenating structure of separable permutations as manifested in \eqref{rec}, we have
\begin{equation}\label{A++cond}
A_n^{+,+}|\{B_1^{+,n}=j\}\stackrel{\text{dist}}{=}A_j^{+,-}|\{B_1^{+,j}=j\}+A_{n-j}^{-,+},\ j=1,\cdots, n-1,\ n\ge3.
\end{equation}
where the random variable on the left hand side is considered under $P^{\text{\rm sep}}_n$,
the random variable $A_j^{+,-}|\{B_1^{+,j}=j\}$ is considered under $P^{\text{\rm sep}}_j$, the random variable $A_{n-j}^{-,+}$
is considered under $P^{\text{\rm sep}}_{n-j}$, and $A_j^{+,-}|\{B_1^{+,j}=j\}$ and $A_{n-j}^{-,+}$ are independent.

To illustrate \eqref{A++cond}, we give three examples of what can occur.  Consider first the  permutation 342178956, which satisfies $B_1^{+,9}(342178956)=4$.
So $n=9$, $j=4$  and $n-j=5$. The length of the longest alternating subsequence that begins and ends with an ascent is six.
Such an alternating subsequence is built from a longest alternating subsequence beginning with an ascent and ending with a descent that appears in
the first $j$ entries of the permutation--- 3421, and then concatenating this with a longest alternating subsequence that begins with a descent and ends with an ascent
in the last $n-j$ entries of the permutation--- 78956. There are two possibilities for the first piece, namely 342 and 341, and there are three
possibilities for the second piece, namely
756 and 856 and 956.
Consider now the permutation 124378956, which satisfies  $B_1^{+,9}(124378956)=1$. So $n=9$,  $j=1$ and $n-j=8$.
The length of the longest alternating subsequence that begins and ends with an ascent is six.
Such an alternating subsequence is built from a longest alternating subsequence beginning with an ascent and ending with a descent that appears in
the first $j$ entries of the permutation--- 1, and   concatenating this with a longest alternating subsequence that begins with a descent and ends with an ascent
in the last $n-j$ entries of the permutation--- 24378956.
 There is one possibility for the first piece, namely 1, and there are three
possibilities for the second piece, namely
43756, 43856 and 43956. Note that for this to work, it was necessary in \eqref{crucialdef} that a singleton be defined as an alternating sequence of the type $(+,-)$.
Finally, consider the permutation 324561789, which satisfies  $B_1^{+,9}(324561789)=6$. So $n=9$, $j=6$ and $n-j=3$.
The length of the longest alternating subsequence that begins and ends with an ascent is four.
Such an alternating subsequence is built from a longest alternating subsequence beginning with an ascent and ending with a descent that appears in
the first $j$ entries of the permutation--- 324561, and then concatenating this with a longest alternating subsequence that begins with a descent and ends with an ascent
in the last $n-j$ entries of the permutation--- 789. There are six possibilities for the first piece, namely 341,351,361,241,251,261, and there are three possibilities for the second piece, namely 7, 8, 9. Note that for this to work, it was necessary in \eqref{crucialdef} that a singleton be defined as  an alternating sequence of the type $(-,+)$.

Now consider \eqref{exp-condj+-} with $A_n^{-,+}$.
From the definition of $A_n^{-,+}$ and the concatenating structure of separable permutations as manifested in \eqref{rec}, we have
\begin{equation}\label{A-+cond}
A_n^{-,+}|\{B_1^{+,n}=j\}\stackrel{\text{dist}}{=}A_j^{-,-}|\{B_1^{+,j}=j\}+A_{n-j}^{-,+},\ j=1,\cdots, n-1,\ n\ge3.
\end{equation}
To illustrate \eqref{A-+cond}, we give two examples  of what can occur.  Consider first the  permutation
342178956, which satisfies $B_1^{+,n}(342178956)=4$.    So
$n=9, j=4$ and $n-j=5$.
The length of the longest alternating subsequence that begins with a descent and ends with an ascent is five.
Such an alternating subsequence is built from a longest alternating subsequence beginning with a descent and ending with a descent that appears in
the first $j$ entries of the permutation--- 3421, and then concatenating this with a longest alternating subsequence that begins with a descent and ends with an ascent
in the last $n-j$ entries of the permutation--- 78956. There are four possibilities for the first piece, namely
32, 42, 31, 41,
and
there are three possibilities for the second piece, namely, 756, 856, 956.
Consider now the permutation 145678923, which satisfies  $B_1^{+,9}(145678923)=1$. So $n=9$,  $j=1$ and $n-j=8$.
The length of the longest alternating subsequence that begins with a descent and ends with an ascent is three.
Such an alternating subsequence is built from a longest alternating subsequence beginning with a descent and ending with a descent that appears in
the first $j$ entries of the permutation--- 1, and then concatenating this with a longest alternating subsequence that begins with a descent and ends with an ascent
in the last $n-j$ entries of the permutation--- 45678923. There is one possibility for the first piece, namely the null set, and there are six possibilities for the second piece, namely
$x23$, with $x\in\{4,5,6,7,8,9\}$. Note that for this to work, it was necessary in \eqref{crucialdef} that a singleton be defined \it not\rm\ to be an alternating sequence of the type $(-,-)$.


Taking expectations in \eqref{A++cond}, and using \eqref{exp-condj+-} along with \eqref{exp-cond+cond} and \eqref{Bprob},
we obtain
$$
\begin{aligned}
&E_n^{\text{sep}}A_n^{+,+}=\frac12c_n^{+,+}+\frac12d_n^{+,+}=\frac{s_1s_{n-1}}{s_n}\big(c_1^{+,-}+\frac12c_{n-1}^{-,+}+\frac12d_{n-1}^{-,+}\big)+\\
&\sum_{j=2}^{n-1}\frac{\frac12s_js_{n-j}}{s_n}\big(c_j^{+,-}+\frac12c_{n-j}^{-,+}+\frac12d_{n-j}^{-,+}\big)+\frac12c_n^{+,+}, n\ge3,
\end{aligned}
$$
or equivalently (noting from \eqref{n=1or2} that
$c^{+,-}_1=1$),
\begin{equation}\label{firstkey}
\frac12d_n^{+,+}=\frac{s_1s_{n-1}}{s_n}\big(1+\frac12c_{n-1}^{-,+}+\frac12d_{n-1}^{-,+}\big)+
\sum_{j=2}^{n-1}\frac{\frac12s_js_{n-j}}{s_n}\big(c_j^{+,-}+\frac12c_{n-j}^{-,+}+\frac12d_{n-j}^{-,+}\big), n\ge3.
\end{equation}
Taking expectations in \eqref{A-+cond}, and using  \eqref{exp-condj+-} along with \eqref{exp-cond+cond} and \eqref{Bprob}, we obtain
$$
\begin{aligned}
&E_n^{\text{sep}}A_n^{-,+}=\frac12c_n^{-,+}+\frac12d_n^{-,+}=\frac{s_1s_{n-1}}{s_n}\big(c_1^{-,-}+\frac12c_{n-1}^{-,+}+\frac12d_{n-1}^{-,+}\big)+\\
&\sum_{j=2}^{n-1}\frac{\frac12s_js_{n-j}}{s_n}\big(c_j^{-,-}+\frac12c_{n-j}^{-,+}+\frac12d_{n-j}^{-,+}\big)+\frac12c_n^{-,+}, n\ge3,
\end{aligned}
$$
or equivalently (noting from \eqref{n=1or2} that
$c^{-,-}_1=0$),
\begin{equation}\label{secondkey}
\frac12d_n^{-,+}=\frac{s_1s_{n-1}}{s_n}\big(\frac12c_{n-1}^{-,+}+\frac12d_{n-1}^{-,+}\big)+
\sum_{j=2}^{n-1}\frac{\frac12s_js_{n-j}}{s_n}\big(c_j^{-,-}+\frac12c_{n-j}^{-,+}+\frac12d_{n-j}^{-,+}\big), n\ge3.
\end{equation}

By Lemma \ref{revcomplemma},  we can substitute in equations \eqref{firstkey} and \eqref{secondkey} so that each of them is given only in terms of
$c_{\cdot}^{-,-}$ and $c_{\cdot}^{+,-}$. Indeed, by Lemma \ref{revcomplemma}, we have
\begin{equation}\label{cdconnection}
d_n^{+,+}=c_n^{-,-},\   c_n^{-,+}=c_n^{+,-}, d_n^{-,+}=c_n^{+,-}.
\end{equation}
Using \eqref{cdconnection} to substitute in \eqref{firstkey} and \eqref{secondkey} and recalling that $s_1=1$, we obtain  the two equations
\begin{equation}\label{keypair}
\begin{aligned}
&\frac12c_n^{-,-}=\frac{s_{n-1}}{s_n}\big(1+c_{n-1}^{+,-}\big)+
\sum_{j=2}^{n-1}\frac{\frac12s_js_{n-j}}{s_n}\big(c_j^{+,-}+c_{n-j}^{+,-}\big), n\ge3;\\
&\frac12c_n^{+,-}=\frac{s_{n-1}}{s_n}c_{n-1}^{+,-}+
\sum_{j=2}^{n-1}\frac{\frac12s_js_{n-j}}{s_n}\big(c_j^{-,-}+c_{n-j}^{+,-}\big), n\ge3.
\end{aligned}
\end{equation}
Multiplying \eqref{keypair} by $2s_nt^n$ and summing over $n$ gives
\begin{equation}\label{keywitht}
\begin{aligned}
&\sum_{n=3}^\infty s_nc_n^{-,-}t^n=2t\sum_{n=3}^\infty s_{n-1}\big(1+c_{n-1}^{+,-}\big)t^{n-1}+
\sum_{n=3}^\infty\Big(\sum_{j=2}^{n-1}s_js_{n-j}\big(c_j^{+,-}+c_{n-j}^{+,-}\big)\Big)t^n;\\
&\sum_{n=3}^\infty s_nc_n^{+,-}t^n=2t\sum_{n=3}^\infty s_{n-1}c_{n-1}^{+,-}t^{n-1}+
\sum_{n=3}^\infty\Big(\sum_{j=2}^{n-1}s_js_{n-j}\big(c_j^{-,-}+c_{n-j}^{+,-}\big)\Big)t^n.
\end{aligned}
\end{equation}
We have
\begin{equation}\label{doublesum}
\begin{aligned}
&\sum_{n=3}^\infty\big(\sum_{j=2}^{n-1}s_js_{n-j}c_j^{+,-}\big)t^n=
\sum_{n=3}^\infty\big(\sum_{j=1}^{n-1}s_js_{n-j}c_j^{+,-}\big)t^n-
t\sum_{n=3}^\infty s_{n-1}t^{n-1}=\\
&\sum_{n=2}^\infty\big(\sum_{j=1}^{n-1}s_js_{n-j}c_j^{+,-}\big)t^n-t\sum_{n=2}^\infty s_{n-1}t^{n-1}=\\
&\big(\sum_{n=1}^\infty s_nc_n^{+,-}t^n\big)
\big(\sum_{n=1}^\infty s_nt^n\big)-t\sum_{n=2}^\infty s_{n-1}t^{n-1},
\end{aligned}
\end{equation}
and similarly,
\begin{equation}\label{doublesumagain}
\begin{aligned}
&\sum_{n=3}^\infty\big(\sum_{j=2}^{n-1}s_js_{n-j}c_{n-j}^{+,-}\big)t^n=
\sum_{n=2}^\infty\big(\sum_{j=1}^{n-1}s_js_{n-j}c_{n-j}^{+,-}\big)t^n-t\sum_{n=2}^\infty s_{n-1}c_{n-1}^{+,-}t^{n-1}=\\
&\big(\sum_{n=1}^\infty s_nt^n\big)
\big(\sum_{n=1}^\infty s_nc_n^{+,-}t^n\big)-t\sum_{n=2}^\infty s_{n-1}c_{n-1}^{+,-}t^{n-1}.
\end{aligned}
\end{equation}
Also, noting from \eqref{n=1or2} that $c_1^{-,-}=0$, we have
\begin{equation}\label{doublesumthirdtime}
\sum_{n=3}^\infty\big(\sum_{j=2}^{n-1}s_js_{n-j}c_j^{-,-}\big)t^n=\sum_{n=2}^\infty\big(\sum_{j=1}^{n-1}s_js_{n-j}c_j^{-,-}\big)t^n=
\big(\sum_{n=1}^\infty s_nc_n^{-,-}t^n\big)\big(\sum_{n=1}^\infty s_nt^n\big).
\end{equation}

Since the conditional probability measure $P_2^{\text{sep}}(\cdot\ |B_1^{+,2}=2)$ gives probability one to the permutation $21$, it follows that
$c_2^{-,-}=2$. From \eqref{n=1or2},     $c_2^{+,-}=1$,  $c_1^{+,-}=1$ and $c_1^{-,-}=0$. Using these facts along with the fact that $s_1=1$, $s_2=2$,  it follows from \eqref{keywitht}-\eqref{doublesumthirdtime} and \eqref{Ggenmean} that
\begin{equation}\label{genfuncequa--}
\begin{aligned}
&G^{-,-}(t)-4t^2=2t(s(t)-t)+2t(G^{+,-}(t)-t)+G^{+,-}(t)s(t)-ts(t)+\\
&G^{+,-}(t)s(t)-tG^{+,-}(t)
\end{aligned}
\end{equation}
and
\begin{equation}\label{genfuncequa+-}
G^{+,-}(t)-2t^2-t=2t(G^{+,-}(t)-t)+G^{-,-}(t)s(t)+
G^{+,-}(t)s(t)-tG^{+,-}(t).
\end{equation}
The equation in \eqref{genfuncequa--} simplifies
to
\begin{equation}\label{G--}
G^{-,-}(t)=\big(2s(t)+t\big)G^{+,-}(t)+ts(t).
\end{equation}
Using \eqref{G--} to substitute for $G^{-,-}(t)$ in \eqref{genfuncequa+-}, and performing some algebra, we obtain
\begin{equation}\label{G+-}
G^{+,-}(t)=\frac{t(1+s^2(t))}{1-t-s(t)-2s^2(t)-ts(t)}.
\end{equation}

From \eqref{sgenfunc}, we have
\begin{equation}\label{s^2}
s^2(t)=\frac12\big(t^2-4t+1-(1-t)\sqrt{t^2-6t+1}\big).
\end{equation}
Using \eqref{sgenfunc} and \eqref{s^2}, we obtain
\begin{equation}\label{denomG+-}
1-t-s(t)-2s^2(t)-ts(t)=-\frac12\big(t^2-6t+1)-\frac12(t-3)\sqrt{t^2-6t+1}
\end{equation}
and
\begin{equation}\label{numeratorG+-}
s^2(t)+1=\frac32-2t+\frac12t^2-\frac{1-t}2\sqrt{t^2-6t+1}.
\end{equation}
Substituting \eqref{denomG+-} and \eqref{numeratorG+-} in \eqref{G+-}, and multiplying the numerator and denominator by $-2$ yields
\begin{equation}\label{G+-again}
G^{+,-}(t)=\frac{t\big(-3+4t-t^2+(1-t)\sqrt{t^2-6t+1}\thinspace\big)}{t^2-6t+1+(t-3)\sqrt{t^2-6t+1}}.
\end{equation}
Writing $-3+4t-t^2=1-(t-2)^2=(1-t+2)(1+t-2)=(t-3)(1-t)$, we have
\begin{equation}\label{simplificationG+-}
\begin{aligned}
&\frac{t\big(-3+4t-t^2+(1-t)\sqrt{t^2-6t+1}\thinspace\big)}{t^2-6t+1+(t-3)\sqrt{t^2-6t+1}}=\\
&t(1-t)\frac{t-3+\sqrt{t^2-6t+1}}{t^2-6t+1+(t-3)\sqrt{t^2-6t+1}}=\frac{t(1-t)}{\sqrt{t^2-6t+1}}.
\end{aligned}
\end{equation}
The formula for $G^{+,-}$ in \eqref{Ggenmeanform} follows from \eqref{G+-again} and \eqref{simplificationG+-}.
The formula for $G^{-,-}$ in \eqref{Ggenmeanform} follows from the formula in \eqref{Ggenmeanform} for $G^{+,-}$, \eqref{G--} and \eqref{sgenfunc}, along with a little algebra.
\end{proof}

\section{Proof of Theorem \ref{mean}}\label{meanproof}
To prove the theorem, it suffices to prove \eqref{meanasymp} with $E_n^\text{sep}A_n^{+,-}$ in place of $E_n^\text{sep}A_n$, since
$A_n(\sigma)-A_n^{+,-}(\sigma)\in\{0,1,2\}$, for all $\sigma\in S_n$.
By Lemma \ref{revcomplemma}, $d^{+,-}=c^{-,+}=c^{+,-}$. From this and
 \eqref{exp-cond+cond},  we have $E_n^\text{sep}A_n^{+,-}=c_n^{+,-}$.
Thus, from \eqref{Ggenmean}, it follows that the coefficient of $t^n$ in the power series for $G^{+,-}(t)$ is
$s_nE_n^\text{sep}A_n^{+,-}$. By
Theorem \ref{Ggenmeanthm} along with Proposition \ref{oneover}, which appears in the appendix, the coefficient of $t^n$ in $G(t)$ is
$a_{n-1}-a_{n-2}$, where $a_n$ satisfies \eqref{oneoverasymp}.
Thus, we have
\begin{equation}\label{almostfinalmean}
\begin{aligned}
&s_nE_n^\text{sep}A_n^{+,-}\sim a_{n-1}-a_{n-2}\sim\frac1{2^\frac54\sqrt\pi n^\frac12}
\big((3-2\sqrt2)^{-n+\frac12}-(3-2\sqrt2)^{-n+\frac32}\big)=\\
&\frac1{2^\frac54\sqrt\pi n^\frac12}
(3-2\sqrt2)^{-n+\frac12}(2\sqrt2-2)=\frac1{2^\frac34\sqrt\pi n^\frac12}
(3-2\sqrt2)^{-n+\frac12}(2-\sqrt2).
\end{aligned}
\end{equation}
By  \eqref{snasympappendix}  in Proposition \ref{snasympprop} in the appendix, we have
\begin{equation*}\label{snasymp}
s_n\sim\frac1{2^\frac34\sqrt\pi n^\frac32}(3-2\sqrt2)^{-n+\frac12}.
\end{equation*}
Using this with  \eqref{almostfinalmean}, we conclude that
$E_n^\text{sep}A_n^{+,-}\sim (2-\sqrt2)n$, which completes the proof of the theorem.
\hfill $\square$

\section{Generating functions related to the second moment of $A_n$}\label{genfuncs2nd}

Define
\begin{equation}\label{C'sD's}
\begin{aligned}
&C_n^{\pm,\pm}=E_n^{\text{sep}}\big((A_n^{\pm,\pm})^2|B_1^{+,n}=n\big);
&D_n^{\pm,\pm}=E_n^{\text{sep}}\big((A_n^{\pm,\pm})^2|B_1^{-,n}=n\big).
\end{aligned}
\end{equation}
Define the generating functions
\begin{equation}\label{Hgenvar}
\begin{aligned}
&H^{+,-}(t)=\sum_{n=1}^\infty s_nC_n^{+,-}t^n;\\
&H^{-,-}(t)=\sum_{n=1}^\infty s_nC_n^{-,-}t^n.
\end{aligned}
\end{equation}

We will prove the following theorem.
\begin{theorem}\label{Hgenvarthm}
\begin{equation}\label{Hgenvarform}
\begin{aligned}
& H^{+,-}(t)=2t^2(1-t)^2\big(t^2-6t+1\big)^{-\frac32}+\frac{-5t^3+8t^2-3t}{t-3}\big(t^2-6t+1\big)^{-\frac12}+\\
&\frac12t^2(1-t)\big(t^2-6t+1\big)^\frac12-\frac12t^2(1-t);\\
&H^{-,-}(t)=H^{+,-}(t)-\frac{-5t^3+8t^2-3t}{t-3}
-\frac12t^2(1-t)^(t^2-6t+1)+\\
&\frac12t^2(1-t)(t^2-6t+1)^\frac12+\frac12t(1-t)-\frac12t(t^2-6t+1)+2t^2(1-t)(t^2-6t+1)^{-\frac12}.
\end{aligned}
\end{equation}
\end{theorem}
\begin{proof}
In several equations below, the notation
$E_n^\text{sep}\big((A_j^{\delta_1,\delta_2}+A_{n-j}^{\delta_3,\delta_4})^2|B_1^{\delta_5,j}=j,B_1^{\delta_6,n-j}=n-j\big)$
is employed, where $\delta_i\in\{+,-\}$, for $i=1,\cdots, 6$. This indicates that
$A_j^{\delta_1,\delta_2}$ and $A_{n-j}^{\delta_3,\delta_4}$ are independent, with $A_j^{\delta_1,\delta_2}$ considered under the measure $E_j^\text{sep}(\cdot\ |B_1^{\delta_5,j}=j)$ and
$A_{n-j}^{\delta_3,\delta_4}$ considered under the measure $E_{n-j}^\text{sep}(\cdot\ |B_1^{\delta_6,n-j}=n-j)$.

We proceed in the manner of the proof of Theorem \ref{mean}.
Using
\eqref{A++cond},
\eqref{Bprob}, \eqref{sym} and \eqref{C'sD's}, and noting from \eqref{n=1or2} that $A_1^{+,-}=1$,
we have
\begin{equation*}
\begin{aligned}
&E_n^\text{sep}(A_n^{+,+})^2=\frac12C_n^{+,+}+\frac12D_n^{+,+}=\\
&\frac{s_1s_{n-1}}{s_n}\Big(\frac12E\big((1+A_{n-1}^{-,+})^2|B_1^{+,n-1}=n-1\big)+
\frac12E\big((1+A_{n-1}^{-,+})^2|B_1^{-,n-1}=n-1\big)\Big)+\\
&\sum_{j=2}^{n-1}\frac12\frac{s_js_{n-j}}{s_n}\frac12E\big((A_j^{+,-}+A_{n-j}^{-,+})^2|B_1^{+,j}=j,B_1^{+,n-j}=n-j\big)+\\
&\sum_{j=2}^{n-1}\frac12\frac{s_js_{n-j}}{s_n}\frac12E\big((A_j^{+,-}+A_{n-j}^{-,+})^2|B_1^{+,j}=j,B_1^{-,n-j}=n-j\big)+\frac12C_n^{+,+},\ n\ge3.
\end{aligned}
\end{equation*}
Expanding the terms in the equation above, cancelling the term $\frac12 C_n^{+,+}$ from both sides and using the notation from \eqref{C'sD's} and \eqref{c'sd's}, we have
\begin{equation*}
\begin{aligned}
&\frac12D_n^{+,+}=\frac{s_1s_{n-1}}{s_n}\big(1+c_{n-1}^{-,+}+\frac12C_{n-1}^{-,+}+d_{n-1}^{-,+}+\frac12D_{n-1}^{-,+}\big)+\\
&\frac14\sum_{j=2}^{n-1}\frac{s_js_{n-j}}{s_n}\big(C_j^{+,-}+2c_j^{+,-}c_{n-j}^{-,+}+C_{n-j}^{-,+}\big)+
\frac14\sum_{j=2}^{n-1}\frac{s_js_{n-j}}{s_n}\big(C_j^{+,-}+2c_j^{+,-}d_{n-j}^{-,+}+D_{n-j}^{-,+}\big),\\
& n\ge3.
\end{aligned}
\end{equation*}
By Lemma \ref{revcomplemma}, $D_n^{+,+}=C_n^{-,-}, D_n^{-,+}=C_n^{+,-}, C_n^{-,+}=C_n^{+,-}, c_n^{-,+}=c_n^{+,-}, d_n^{-,+}=c_n^{+,-}$, for all $n$. Making these substitutions in the equation above,  multiplying
both sides by $2s_n$ and recalling that $s_1=1$ gives
\begin{equation}\label{varrecur--}
s_nC_n^{-,-}=s_{n-1}\big(2+4c_{n-1}^{+,-}+2C_{n-1}^{+,-}\big)+
\sum_{j=2}^{n-1}s_js_{n-j}\big(C_j^{+,-}+2c_j^{+,-}c_{n-j}^{+,-}+C_{n-j}^{+,-}\big), \ n\ge3.
\end{equation}

To derive a second recursion equation similar to the one in \eqref{varrecur--} we consider $A_n^{-,+}$.
Using
\eqref{A-+cond},
\eqref{Bprob}, \eqref{sym} and \eqref{C'sD's} and noting from \eqref{n=1or2} that $A_1^{-,-}=0$,
we have
\begin{equation*}
\begin{aligned}
&E_n^\text{sep}(A_n^{-,+})^2=\frac12C_n^{-,+}+\frac12D_n^{-,+}=\\
&\frac{s_1s_{n-1}}{s_n}\Big(\frac12E\big((A_{n-1}^{-,+})^2|B_1^{+,n-1}=n-1\big)+
\frac12E\big((A_{n-1}^{-,+})^2|B_1^{-,n-1}=n-1\big)\Big)+\\
&\sum_{j=2}^{n-1}\frac12\frac{s_js_{n-j}}{s_n}\frac12E\big((A_j^{-,-}+A_{n-j}^{-,+})^2|B_1^{+,j}=j,B_1^{+,n-j}=n-j\big)+\\
&\sum_{j=2}^{n-1}\frac12\frac{s_js_{n-j}}{s_n}\frac12E\big((A_j^{-,-}+A_{n-j}^{-,+})^2|B_1^{+,j}=j,B_1^{-,n-j}=n-j\big)+\frac12C_n^{-,+},\ n\ge3.
\end{aligned}
\end{equation*}
Expanding the terms in the equation above, cancelling the term $\frac12 C_n^{-,+}$ from both sides and using the notation from \eqref{C'sD's} and \eqref{c'sd's}, we have
\begin{equation*}
\begin{aligned}
&\frac12D_n^{-,+}=\frac{s_1s_{n-1}}{s_n}\big(\frac12C_{n-1}^{-,+}+\frac12D_{n-1}^{-,+}\big)+
\frac14\sum_{j=2}^{n-1}\frac{s_js_{n-j}}{s_n}\big(C_j^{-,-}+2c_j^{-,-}c_{n-j}^{-,+}+C_{n-j}^{-,+}\big)+\\
&\frac14\sum_{j=2}^{n-1}\frac{s_js_{n-j}}{s_n}\big(C_j^{-,-}+2c_j^{-,-}d_{n-j}^{-,+}+D_{n-j}^{-,+}\big),\
n\ge3.
\end{aligned}
\end{equation*}
By Lemma \ref{revcomplemma}, $D_n^{-,+}=C_n^{+,-},  C_n^{-,+}=C_n^{+,-}, c_n^{-,+}=c_n^{+,-}, d_n^{-,+}=c_n^{+,-}$, for all $n$. Making these substitutions in the equation above,  multiplying
both sides by $2s_n$ and recalling that $s_1=1$ gives
\begin{equation}\label{varrecur-+}
s_nC_n^{+,-}=2s_{n-1}C_{n-1}^{+,-}+
\sum_{j=2}^{n-1}s_js_{n-j}\big(C_j^{-,-}+2c_j^{-,-}c_{n-j}^{+,-}+C_{n-j}^{+,-}\big),\ n\ge3.
\end{equation}

Multiplying \eqref{varrecur--} and \eqref{varrecur-+} by $t^n$ and summing over $n$ gives
\begin{equation}\label{key2ndt1}
\begin{aligned}
&\sum_{n=3}^\infty s_nC_n^{-,-}t^n=2t\sum_{n=3}^\infty s_{n-1}t^{n-1}+4t\sum_{n=3}^\infty s_{n-1}c_{n-1}^{+,-}t^{n-1}+
2t\sum_{n=3}^\infty s_{n-1}C_{n-1}^{+,-}t^{n-1}+\\
&\sum_{n=3}^\infty\big(\sum_{j=2}^{n-1}s_js_{n-j}C_j^{+,-}\big)t^n+\sum_{n=3}^\infty\big(\sum_{j=2}^{n-1}s_js_{n-j}C_{n-j}^{+,-}\big)t^n+
2\sum_{n=3}^\infty\big(\sum_{j=2}^{n-1}s_js_{n-j}c_j^{+,-}c_{n-j}^{+,-}\big)t^n
\end{aligned}
\end{equation}
and
\begin{equation}\label{key2ndt2}
\begin{aligned}
&\sum_{n=3}^\infty s_nC_n^{+,-}t^n=
2t\sum_{n=3}^\infty s_{n-1}C_{n-1}^{+,-}t^{n-1}+
\sum_{n=3}^\infty\big(\sum_{j=2}^{n-1}s_js_{n-j}C_j^{-,-}\big)t^n+\\
&\sum_{n=3}^\infty\big(\sum_{j=2}^{n-1}s_js_{n-j}C_{n-j}^{+,-}\big)t^n+
2\sum_{n=3}^\infty\big(\sum_{j=2}^{n-1}s_js_{n-j}c_j^{-,-}c_{n-j}^{+,-}\big)t^n.
\end{aligned}
\end{equation}

We have $s_2=2$. By \eqref{n=1or2}, $C_1^{+,-}=1$ and $C_1^{-,-}=0$. By the explanation in the paragraph after  \eqref{doublesumthirdtime},
we have
$C_2^{-,-}=4$ and $C_2^{+,-}=1$.
The first and second double sums on the second line of \eqref{key2ndt1} satisfy \eqref{doublesum} and \eqref{doublesumagain} respectively  with $c_{\cdot}^{+,-}$ replaced by
$C_{\cdot}^{+,-}$. The third double sum on the second line of \eqref{key2ndt1} satisfies similarly,
 $$
 \begin{aligned}
& \sum_{n=3}^\infty\big(\sum_{j=2}^{n-1}s_js_{n-j}c_j^{+,-}c_{n-j}^{+,-}\big)t^n=\\
& \big(\sum_{n=1}^\infty s_nc_n ^{+,-}t^n\big)^2
-t\sum_{n=2}^\infty s_{n-1}c_{n-1}^{+,-}t^{n-1}.
 \end{aligned}
 $$
 The double sum on the first line of \eqref{key2ndt2} satisfies \eqref{doublesumthirdtime} with with $c_{\cdot}^{-,-}$ replaced by
$C_{\cdot}^{-,-}$ and the first double sum on the second line of \eqref{key2ndt2} satisfies
\eqref{doublesumagain} with $c_{\cdot}^{+,-}$ replaced by $C_{\cdot}^{+,-}$.
The second double sum on the second line of \eqref{key2ndt2} satisfies similarly
$$
\sum_{n=3}^\infty\big(\sum_{j=2}^{n-1}s_js_{n-j}c_j^{-,-}c_{n-j}^{+,-}\big)t^n=\big(\sum_{n=1}^\infty s_nc_n ^{-,-}t^n\big)\big(\sum_{n=1}^\infty s_nc_n ^{+,-}t^n\big).
$$
Using the above facts  with \eqref{Hgenvar}, we obtain
from \eqref{key2ndt1} and \eqref{key2ndt2}
\begin{equation}\label{Hgenfuncequa--}
\begin{aligned}
&H^{-,-}(t)-8t^2=2t(s(t)-t)+4t\big(G^{+,-}(t)-t\big)+2t\big(H^{+,-}(t)-t\big)+\\
&s(t)H^{+,-}(t)-ts(t)+s(t)H^{+,-}(t)-tH^{+,-}(t)+2(G^{+,-}(t))^2-2tG^{+,-}(t)
\end{aligned}
\end{equation}
and
\begin{equation}\label{Hgenfuncequa+-}
\begin{aligned}
&H^{+,-}(t)-2t^2-t=2t\big(H^{+,-}(t)-t\big)+s(t)H^{-,-}(t)+\\
&s(t)H^{+,-}(t)-tH^{+,-}(t)+2G^{-,-}(t)G^{+,-}(t).
\end{aligned}
\end{equation}
The  equation in \eqref{Hgenfuncequa--} simplifies to
\begin{equation}\label{H--}
H^{-,-}(t)=\big(2s(t)+t\big)H^{+,-}(t)+ts(t)+2tG^{+,-}(t)+2(G^{+,-}(t))^2.
\end{equation}
From \eqref{sgenfunc}, $2s(t)+t=1-\sqrt{t^2-6t+1}$. Using this,
the formula for $H^{-,-}$   in \eqref{Hgenvarform}
follows from  \eqref{H--} after substituting  for $G^{+,-}$ from  \eqref{Ggenmeanform}
 and performing some algebra.

Using \eqref{H--} to substitute for $H^{-,-}(t)$ in \eqref{Hgenfuncequa+-}, and then solving  for $H^{+,-}(t)$, one obtains after some algebra,
\begin{equation}\label{H+-}
H^{+,-}(t)=\frac{t+ts^2(t)+2G^{-,-}(t)G^{+,-}(t)+2ts(t)G^{+,-}(t)+2s(t)(G^{+,-}(t))^2}{1-t-s(t)-2s^2(t)-ts(t)}.
\end{equation}
Substituting for $G^{+,-}(t)$ and $G^{-,-}(t)$ from \eqref{Ggenmeanform},  substituting for $s(t)$ from \eqref{sgenfunc} and substituting for $s^2(t)$ from \eqref{s^2}, and performing a lot of algebra, one finds that
the numerator in \eqref{H+-} satifies
\begin{equation}\label{numeratorH+-}
\begin{aligned}
&t+ts^2(t)+2G^{-,-}(t)G^{+,-}(t)+2ts(t)G^{+,-}(t)+2s(t)(G^{+,-}(t))^2=\\
&\frac52t^3-4t^2+\frac32t-\frac{t(1-t)}2\sqrt{t^2-6t+1}-\frac{t^2(1-t)^2}{\sqrt{t^2-6t+1}}-\frac{t^2(1-t)^2(t-3)}{t^2-6t+1}.
\end{aligned}
\end{equation}

Let
\begin{equation}\label{X}
X=\sqrt{t^2-6t+1}
\end{equation}
 and $Y=t-3$.
 The denominator of $H^{+,-}(t)$ in \eqref{H+-} was calculated in
\eqref{denomG+-}, and is equal to $-\frac12(X^2+YX)$. We have
\begin{equation}\label{XY}
\begin{aligned}
&\frac1{X^2+YX}=\frac1Y(\frac1X-\frac1{X+Y})=\\
&\frac1Y(\frac1X+\frac{Y-X}{X^2-Y^2})=\frac1{t-3}\Big(\frac1{\sqrt{t^2-6t+1}}+\frac{t-3-\sqrt{t^2-6t+1}}{t^2-6t+1-(t-3)^2}\Big)=\\
&\frac1{t-3}\Big(\frac1{\sqrt{t^2-6t+1}}-\frac{t-3-\sqrt{t^2-6t+1}}8\Big)=\\
&\frac1{(t-3)\sqrt{t^2-6t+1}}-\frac18+\frac{\sqrt{t^2-6t+1}}{8(t-3)}.
\end{aligned}
\end{equation}
From  \eqref{H+-}--\eqref{XY}, we obtain
\begin{equation}\label{H+-almostexplic}
\begin{aligned}
&H^{+,-}(t)=-2\Big(\frac1{(t-3)\sqrt{t^2-6t+1}}-\frac18+\frac{\sqrt{t^2-6t+1}}{8(t-3)}\Big)\times\\
&\Big((\frac52t^3-4t^2+\frac32t)-\frac{t(1-t)}2\sqrt{t^2-6t+1}-\frac{t^2(1-t)^2}{\sqrt{t^2-6t+1}}-\frac{t^2(1-t)^2(t-3)}{t^2-6t+1}\Big).
\end{aligned}
\end{equation}

Denote the three expressions in the first parenthetical factor in \eqref{H+-almostexplic} by $\gamma_i, i=1,2,3$, and denote the four expressions in the second parenthetical  factor in
 \eqref{H+-almostexplic} by $\beta_j, j=1,2,3,4$.
 We multiply out the right hand side of \eqref{H+-almostexplic} in the order of the following double sum:
 $-2\sum_{j=1}^4\sum_{i=1}^3\gamma_i\beta_j$. We have
 $\gamma_2\beta_3+ \gamma_3\beta_4=0$. Also, after some algebra, one finds that
 $\big(\gamma_1\beta_3+\gamma_2\beta_4\big)=\frac{t^2(1-t)^2}{8(t-3)}$, which cancels with the term $\gamma_3\beta_3$.
 Writing down the other seven terms in the order of the double sum above,   and recalling the definition of $X$ in \eqref{X}, we obtain
\begin{equation}\label{H+-explic}
\begin{aligned}
&H^{+,-}(t)=\frac{-5t^3+8t^2-3t}{t-3}\frac1X+\frac{5t^3-8t^2+3t}8+\frac{-5t^3+8t^2-3t}{8(t-3)}X+\\
&\frac{t(1-t)}{t-3}-\frac{t(1-t)}8X+\frac{t(1-t)}{8(t-3)}(t^2-6t+1)+
2t^2(1-t)^2\frac1{X^3}.
\end{aligned}
\end{equation}
We have
\begin{equation}\label{alg1}
\begin{aligned}
&\frac{t(1-t)}{t-3}+\frac{t(1-t)}{8(t-3)}(t^2-6t+1)+\frac{5t^3-8t^2+3t}8=\\
&\frac18t(1-t)(t-3)-\frac18(-3t+8t^2-5t^3)=-\frac12t^2(1-t)
\end{aligned}
\end{equation}
and
\begin{equation}\label{alg2}
 \frac{-5t^3+8t^2-3t}{8(t-3)}-\frac{t(1-t)}8=\frac12t^2(1-t).
\end{equation}
Now
\eqref{Hgenvarform} for $H^{+,-}(t)$ follows from \eqref{H+-explic}-\eqref{alg2}.
\end{proof}

\section{Proof of Theorem \ref{var}}\label{varproof}
We first proof part (i). Then we make minor additions to that proof to obtain part (ii).
By \eqref{A+-distributions}, it suffices to consider $A_n^{+,-}$ for part (i).
By \eqref{sym} and \eqref{C'sD's}, we have
$E_n^{\text{sep}}(A_n^{+,-})^2=\frac12C_n^{+,-}+\frac12D_n^{+,-}$, and by
 Lemma \ref{revcomplemma}, $D_n^{+,-}=C_n^{-,+}=C_n^{+,-}$. Therefore,
 $E_n^{\text{sep}}(A_n^{+,-})^2=C_n^{+,-}$.
Thus
$t^n[H^{+,-}]$, the coefficient of $t^n$ in the power series expansion of $H^{+,-}$, is equal to  $s_nE_n^{\text{sep}}(A_n^{+,-})^2$. As noted in
section \ref{meanproof}, $t^n[G^{+,-}]$, the coefficient of $t^n$ in the power series expansion
of $G^{+,-}$, is equal to $s_nE_n^{\text{sep}}A_n^{+,-}$.
Thus
\begin{equation}\label{variance}
\text{Var}^{\text{sep}}_n(A_n^{+,-})=E_n^{\text{sep}}(A_n^{+,-})^2-\big(E_n^{\text{sep}}A_n^{+,-}\big)^2=\frac{t^n[H^{+,-}]}{s_n}-\big(\frac{t^n[G^{+,-}]}{s_n}\big)^2.
\end{equation}
Recall that $r_1=3-2\sqrt2$ denotes the smaller of  the two roots of $t^2-6t+1$.
For some of  the calculations below, it will be convenient to write $r_1$ instead of $3-2\sqrt2$ at certain points.

We first consider the asymptotic behavior of $E_n^{\text{sep}}A_n^{+,-}=\frac{t^n[G^{+,-}]}{s_n}$.
From \ref{Ggenmeanform} in Theorem \ref{Ggenmeanthm} and from   Proposition \ref{oneover} in the appendix, we have
\begin{equation}\label{anan-1}
t^n[G^{+,-}]=a_{n-1}-a_{n-2},
\end{equation}
where $a_n$ satisfies \eqref{oneoverasymprefine}.
From \eqref{oneoverasymprefine},
\begin{equation}\label{ans}
\begin{aligned}
&a_{n-i}=\frac{(3-2\sqrt2)^{-n-\frac12}}{\sqrt\pi}\Big(
\frac1{2^\frac54}\frac{(3-2\sqrt2)^i}{(n-i)^\frac12}+
\frac{3-4\sqrt2}{32\cdot2^\frac34}\frac{(3-2\sqrt2)^i}{(n-i)^\frac32}+o(\frac1{n^\frac32})\Big), i=1,2.
\end{aligned}
\end{equation}
Thus,
\begin{equation}\label{ansagain}
\begin{aligned}
&2^\frac34n^\frac32a_{n-i}=\frac{(3-2\sqrt2)^{-n-\frac12}}{\sqrt\pi}\Big(
\frac{(3-2\sqrt2)^i}{\sqrt2}n(\frac n{n-i})^\frac12+\\
&\frac{(3-4\sqrt2)(3-2\sqrt2)^i}{32}+o(1)\Big), i=1,2.
\end{aligned}
\end{equation}
Using \eqref{anan-1} and \eqref{ansagain}, along  with \eqref{snasymprefine} in the appendix, which gives the asymptotic behavior of $s_n$, we have
\begin{equation}\label{meanA+-}
\begin{aligned}
&E_n^{\text{sep}}A_n^{+,-}=\frac{t^n[G^{+,-}]}{s_n}=\frac1{3-2\sqrt2}\Big(1+\frac{24-9\sqrt2}{32}\frac1n+o(\frac1n)\Big)^{-1}\times\\
&\sum_{i=1}^2(-1)^{i-1}\Big(\frac{(3-2\sqrt2)^i}{\sqrt2}n(\frac n{n-i})^\frac12+
\frac{(3-4\sqrt2)(3-2\sqrt2)^i}{32}+o(1)\Big)=\\
&\frac1{3-2\sqrt2}\Big(1-\frac{24-9\sqrt2}{32}\frac1n+o(\frac1n)\Big)\times\\
&\sum_{i=1}^2(-1)^{i-1}\Big(\frac{(3-2\sqrt2)^i}{\sqrt2}n(1+\frac i{2n})+
\frac{(3-4\sqrt2)(3-2\sqrt2)^i}{32}+o(1)\Big)=\\
&(2-\sqrt2)n-\frac14(3-2\sqrt2)+o(1),
\end{aligned}
\end{equation}
where the last equality follows after some algebra.
This proves \eqref{+-meanrefine}.
From \eqref{meanA+-} we obtain
\begin{equation}\label{firstmomsq}
\big(E_n^{\text{sep}}A_n^{+,-}\big)^2=\big(\frac{t^n[G^{+,-}]}{s_n}\big)^2=(6-4\sqrt2)n^2-\frac{10-7\sqrt2}2n+o(n).
\end{equation}

Now we turn to the asymptotic behavior of
$E_n^{\text{sep}}(A_n^{+,-})^2=\frac{t^n[H^{+,-}]}{s_n}$.
Propositions
 \ref{oneover} and \ref{-32} in the appendix give the asymptotic behavior of the coefficient of $t^n$ in  $\big(t^2-6t+1\big)^{-\frac12}$ and in $\big(t^2-6t+1\big)^{-\frac32}$ respectively, and
 Proposition \ref{snasympprop} gives the asymptotic behavior of the coefficient of $t^n$ in $-\frac12 \big(t^2-6t+1\big)^\frac12$.
By these propositions and \eqref{Hgenvarform} in Theorem \ref{Hgenvarthm}, which gives the formula for $H^{+,-}(t)$, the leading order contribution to the coefficient of $t^n$ in $H^{+,-}(t)$ comes from the leading order contribution to the coefficient of $t^n$ in the term
$2t^2(1-t)^2\big(t^2-6t+1\big)^{-\frac32}$, while the next order contribution to the coefficient  of $t^n$ in $H^{+,-}(t)$ comes from
the next order contribution to the coefficient of $t^n$ in $2t^2(1-t)^2\big(t^2-6t+1\big)^{-\frac32}$ and also from the leading order contribution
to the coefficient of $t^n$ in the term $\frac{-5t^3+8t^2-3t}{t-3}\big(t^2-6t+1\big)^{-\frac12}$.

We begin with finding the leading order contribution in the coefficient of $t^n$ in  $\frac{-5t^3+8t^2-3t}{t-3}\big(t^2-6t+1\big)^{-\frac12}$.
Recalling from Proposition \ref{oneover} that $a_n$ denotes the coefficient of $t^n$ in the power series for $(t^2-6t+1)^{-\frac12}$,
we write
\begin{equation}\label{t-3denom}
\frac1{(t-3)\sqrt{t^2-6t+1}}=-\frac13\big(\sum_{n=0}^\infty(\frac t3)^n\big)\big(\sum_{n=0}^\infty a_nt^n\big)=
-\frac13\sum_{n=0}^\infty\Big(\sum_{j=0}^n(\frac13)^ja_{n-j}\Big)t^n.
\end{equation}
We write
\begin{equation}\label{breakup}
\sum_{j=0}^n(\frac13)^ja_{n-j}=\sum_{j=0}^{C\log n}(\frac13)^ja_{n-j}+\sum_{j=C\log n+1}^n(\frac13)^ja_{n-j},
\end{equation}
where $C$ is chosen so that
\begin{equation}\label{Cexp}
\sum_{C\log n+1}^\infty(\frac{3-2\sqrt2}3)^j=o(n^{-\frac12}).
\end{equation}
Considering the asymptotic behavior of $a_n$ in \eqref{oneoverasymp} in the appendix, we write
\begin{equation}\label{123sec6}
\begin{aligned}
&\sum_{j=0}^{C\log n}(\frac13)^ja_{n-j}=\frac{(3-2\sqrt2)^{-n-\frac12}}{2^\frac54\sqrt\pi n^\frac12}\sum_{j=0}^{C\log n}(\frac{3-2\sqrt2}3)^j+\\
&\frac{(3-2\sqrt2)^{-n-\frac12}}{2^\frac54\sqrt\pi}\sum_{j=0}^{C\log n}(\frac{3-2\sqrt2}3)^j\big(\frac1{(n-j)^\frac12}-\frac1{n^\frac12}\big)+\\
&\sum_{j=0}^{C\log n}(\frac13)^j\Big(a_{n-j}-\frac{(3-2\sqrt2)^{-n+j-\frac12}}{2^\frac54\sqrt\pi (n-j)^\frac12}\Big).
\end{aligned}
\end{equation}
From \eqref{123sec6} and \eqref{oneoverasymprefine} in the appendix, and the equality
$\sum_{j=0}^\infty(\frac{3-2\sqrt2}3)^j=\frac3{2\sqrt2}$, it follows that
\begin{equation}\label{est-12termH}
\sum_{j=0}^{C\log n}(\frac13)^ja_{n-j}=\frac{(3-2\sqrt2)^{-n-\frac12}}{2^\frac54\sqrt\pi n^\frac12}\big(\frac3{2\sqrt2}+o(1)\big).
\end{equation}
From \eqref{oneoverasymp}, there exits a $C_1>0$ such that
$a_k\le C_1\frac{(3-2\sqrt2)^{-k}}{k^\frac12}$, $k\in\mathbb{N}$. Using this with \eqref{Cexp}, we obtain
\begin{equation}\label{tailsec6}
\sum_{j=C\log n+1}^n(\frac13)^ja_{n-j}\le C_1(3-2\sqrt2)^{-n}\sum_{j=C\log n+1}^\infty(\frac{3-2\sqrt2}3)^j=o(\frac{(3-2\sqrt2)^{-n}}{n^\frac12}).
\end{equation}
From \eqref{t-3denom}, \eqref{breakup}, \eqref{est-12termH} and \eqref{tailsec6}, we conclude that
the coefficient  $-\frac13\sum_{j=0}^n(\frac13)^ja_{n-j}$ of $t^n$ in the power series for
$\frac1{(t-3)\sqrt{t^2-6t+1}}$ satisfies
\begin{equation}\label{t-3coeff}
-\frac13\sum_{j=0}^n(\frac13)^ja_{n-j}\sim-\frac13\frac{(3-2\sqrt2)^{-n-\frac12}}{2^{\frac54}\sqrt\pi n^\frac12}\frac3{2\sqrt2}=
-\frac{(3-2\sqrt2)^{-n-\frac12}}{4\cdot2^{\frac34}\sqrt\pi n^\frac12}.
\end{equation}
From the above calculations, we conclude that
\begin{equation}\label{I1}
\begin{aligned}
&\text{the leading order contribution in the coefficient of }\ t^n\ \text{in}\\
&\frac{-3t+8t^2-5t^3}{t-3}\big(t^2-6t+1\big)^{-\frac12}\ \text{is}\ \frac{(3-2\sqrt2)^{-n-\frac12}}{4\cdot2^{\frac34}\sqrt\pi }\big(\frac{3r_1}{(n-1)^\frac12}-\frac{8r_1^2}{(n-2)^\frac12}+\frac{5r_1^3}{(n-3)^\frac12}\big)=\\
&\frac{(3-2\sqrt2)^{-n+\frac12}}{4\cdot2^{\frac34}\sqrt\pi}\big(\frac3{(n-1)^\frac12}-\frac{8r_1}{(n-2)^\frac12}+\frac{5r_1^2}{(n-3)^\frac12}\big)
:=I_1.
\end{aligned}
\end{equation}

We now consider $2t^2(1-t)^2\big(t^2-6t+1\big)^{-\frac32}$.
Writing $2t^2(1-t)^2=2t^2-4t^3+2t^4$, it follows from Proposition \ref{-32} that
the two leading orders in the  contribution to the coefficient of $t^n$ in $2t^2(1-t)^2\big(t^2-6t+1\big)^{-\frac32}$
 are contained  in
$2\alpha_{n-2}-4\alpha_{n-3}+2\alpha_{n-4}$,
where $\alpha_n$ satisfies
\eqref{-32asymp}.
Thus, from \eqref{-32asymp},
\begin{equation}\label{I2}
\begin{aligned}
&\text{the  leading order  contribution to the coefficient of}\ t^n\ \text{in}\\
&2t^2(1-t)^2\big(t^2-6t+1\big)^{-\frac32}\
\text{is given by}\\
&\frac{(3-2\sqrt2)^{-n+\frac12}}{4\cdot2^\frac34\sqrt\pi}
\Big(2(n-2)^\frac12-4r_1(n-3)^\frac12+2r_1^2(n-4)^\frac12\Big):=I_2,
\end{aligned}
\end{equation}
 and
 \begin{equation}\label{I3}
\begin{aligned}
&\text{the second leading order of contribution to the coefficient of}\ t^n\ \text{in}\\
&2t^2(1-t)^2\big(t^2-6t+1\big)^{-\frac32}\
\text{is given by}\\
&\frac{(24-9\sqrt2)(3-2\sqrt2)^{-n+\frac12}}{128\cdot2^\frac34\sqrt\pi}\Big(2(n-2)^{-\frac12}-4r_1(n-3)^{-\frac12}+2r_1^2(n-4)^{-\frac12}\Big):=I_3.
\end{aligned}
\end{equation}

By \eqref{snasymprefine} in the appendix,
 \begin{equation}\label{I4}
\begin{aligned}
& \text{the first two leading order terms in}\ s_n\ \text{are}\\
&\frac1{\sqrt\pi}(3-2\sqrt2)^{-n+\frac12}
\Big(\frac1{2^\frac34}n^{-\frac32}+\frac{12\sqrt2-9}{2^\frac14\cdot32}n^{-\frac52}\Big):=I_4.
\end{aligned}
\end{equation}
 From \eqref{I1}-\eqref{I4}, we conclude that
 \begin{equation}\label{conclude}
 \begin{aligned}
& \text{the first two leading order terms in}\
 E_n^{\text{sep}}(A_n^{+,-})^2=\frac{t^n[H^{+,-}]}{s_n}\\
 & \text{are contained in}\
 \frac{I_1+I_2+I_3}{I_4}\ \text{and are of the form}\ An^2+Bn.
\end{aligned}
\end{equation}

 From \eqref{I1} and \eqref{I4},
 \begin{equation}\label{I1I4}
\begin{aligned}
&\frac{I_1}{I_4}=\frac{\frac1{4\cdot 2^\frac34}\big(\frac3{(n-1)^\frac12}-\frac{8r_1}{(n-2)^\frac12}+\frac{5r_1^2}{(n-3)^\frac12}\big)}{\frac1{2^\frac34}n^{-\frac32}+\frac{12\sqrt2-9}{2^\frac14\cdot32}n^{-\frac52}}=\\
&\frac{\frac n4\big(3(\frac n{n-1})^\frac12-8r_1(\frac n{n-2})^\frac12+5r_1^2(\frac n{n-3})^\frac12\big)}{1+\frac{(12\sqrt2-9)\sqrt2}{32}\frac1n}=\frac{3-8r_1+5r_1^2}4n+O(1)=\\
&\frac{3-8(3-2\sqrt2)+5(3-2\sqrt2)^2}4n+o(n)=(16-11\sqrt2)n+o(n).
\end{aligned}
 \end{equation}
 From \eqref{I2} and \eqref{I4},
 \begin{equation}\label{I2I4}
\begin{aligned}
&\frac{I_2}{I_4}=\frac{\frac1{4\cdot2^\frac34}\big(2(n-2)^\frac12-4r_1(n-3)^\frac12+2r_1^2(n-4)^\frac12\big)}{\frac1{2^\frac34}n^{-\frac32}+\frac{12\sqrt2-9}{2^\frac14\cdot32}n^{-\frac52}}=\\
&\frac{\frac12(\frac{n-2}n)^\frac12-r_1(\frac {n-3}n)^\frac12+\frac12r_1^2(\frac{n-4}n)^\frac12}{1+\frac{(24-9\sqrt2)}{32n}}n^2=\\
&\Big(\frac12(1-\frac1n)-r_1(1-\frac3{2n})+\frac12r_1^2(1-\frac2n)+O(\frac1{n^2})\Big)\Big(1+\frac{9\sqrt2-24}{32n}+O(\frac1{n^2})\Big)n^2=\\
&(6-4\sqrt2)n^2+(-\frac{79}4+\frac{219}{16}\sqrt2)n+O(1).
\end{aligned}
\end{equation}
From \eqref{I3} and \eqref{I4},
\begin{equation}\label{I3I4}
\begin{aligned}
&\frac{I_3}{I_4}=\frac{\frac{24-9\sqrt2}{128\cdot2^\frac34}\big(2(n-2)^{-\frac12}-4r_1(n-3)^{-\frac12}+2r_1^2(n-4)^{-\frac12}\big)}
{\frac1{2^\frac34}n^{-\frac32}+\frac{12\sqrt2-9}{32\cdot 2^\frac14}n^{-\frac52}}=\\
&\frac{\frac{24-9\sqrt2}{64}\Big((\frac n{n-2})^\frac12-2r_1(\frac n{n-3})^\frac12+r_1^2(\frac n{n-4})^\frac12\Big)}{1+\frac{24-9\sqrt2}{32n}}n=\\
&\frac{24-9\sqrt2}{64}\big(1-2r_1+r_1^2\big)n+O(1)=\frac{24-9\sqrt2}{64}(r_1-1)^2n+O(1)=\\
&(\frac{24-9\sqrt2}{64})(2-2\sqrt2)^2n+O(1)=\big(\frac{27}4-\frac{75}{16}\sqrt2\big)n+O(1).
\end{aligned}
\end{equation}

From \eqref{conclude}-\eqref{I3I4} we conclude that
\begin{equation}\label{finalform}
\begin{aligned}
&E_n^{\text{sep}}(A_n^{+,-})^2=\frac{t^n[H^{+,-}]}{s_n}=\\
&(6-4\sqrt2)n^2+\Big(16-11\sqrt2-\frac{79}4+\frac{219}{16}\sqrt2+\frac{27}4-\frac{75}{16}\sqrt2\Big)n+O(1)=\\
&(6-4\sqrt2)n^2+(3-2\sqrt2)n+O(1).
\end{aligned}
\end{equation}
Using \eqref{finalform} with   \eqref{firstmomsq} and \eqref{variance}, we conclude that
\begin{equation}\label{finally}
\begin{aligned}
\text{Var}^{\text{sep}}_n(A_n^{+,-})=(3-2\sqrt2+\frac{10-7\sqrt2}2)n=\frac{16-11\sqrt2}2n\approx0.444\thinspace n,
\end{aligned}
\end{equation}
which proves
\eqref{varasymp}.

We now consider part (ii). By \eqref{A+-distributions}, it suffices to consider $A_n^{-,-}$.
We first consider the asymptotic behavior of
$E_n^{\text{sep}}A_n^{-,-}=\frac{t^n[G^{-,-}]}{s_n}$.
From  \eqref{Ggenmeanform}, we have
\begin{equation}\label{G--from+-}
t^n[G^{-,-}]=t^n[G^{+,-}]-t^n[\frac12t(t^2-6t+1)^\frac12],\ n\ge3.
\end{equation}
By
\eqref{bn} and \eqref{bnasymp} in the appendix, we have
\begin{equation}\label{pieceofG--asymp}
t^n[\frac12t(t^2-6t+1)^\frac12]=\frac12b_{n-1}\sim-\frac1{2^\frac34\sqrt\pi}(3-2\sqrt2)^{-n+\frac32}\thinspace\frac1{n^\frac32}.
\end{equation}
From   \eqref{pieceofG--asymp}
and \eqref{snasympappendix} in the appendix, we obtain
\begin{equation}\label{pieceofmeanA--}
-\frac{t^n[\frac12t(t^2-6t+1)^\frac12]}{s_n}=3-2\sqrt2+o(1).
\end{equation}
From \eqref{G--from+-}, \eqref{meanA+-} and
\eqref{pieceofmeanA--}, we conclude that
\begin{equation}\label{meanA--}
E_n^{\text{sep}}A_n^{-,-}=\frac{t^n[G^{-,-}]}{s_n}=
(2-\sqrt2)n+\frac34(3-2\sqrt2)+o(1),
\end{equation}
which proves \eqref{--meanrefine}.

We now turn to the variance. We have
\begin{equation}\label{varianceagain}
\text{Var}^{\text{sep}}_n(A_n^{-,-})=E_n^{\text{sep}}(A_n^{-,-})^2-\big(E_n^{\text{sep}}A_n^{-,-}\big)^2=\frac{t^n[H^{-,-}]}{s_n}-\big(\frac{t^n[G^{-,-}]}{s_n}\big)^2.
\end{equation}
From \eqref{meanA--}, we have
\begin{equation}\label{firstmomsqagain}
(E_n^{\text{sep}}A_n^{-,-})^2=\big(\frac{t^n[G^{-,-}]}{s_n}\big)^2=(6-4\sqrt2)n^2+\frac{30-21\sqrt2}2n+o(n).
\end{equation}
From  \eqref{Hgenvarform}, we have
\begin{equation}\label{H--from+-}
t^n[H^{-,-}]=t^n[H^{+,-}]+t^n[2t^2(1-t)(t^2-6t+1)^{-\frac12}]+\ \text{lower order terms}.
\end{equation}
From Proposition  \ref{oneover} in the appendix,
$t^n[2t^2(1-t)(t^2-6t+1)^{-\frac12}]=2(a_{n-2}-a_{n-3})$. Thus,  using   \eqref{oneoverasymp} from the appendix, or alternatively, using
 \eqref{almostfinalmean} with $n$ replaced by $n-1$,
we have
\begin{equation}\label{pieceofHasymp}
t^n[2t^2(1-t)(t^2-6t+1)^{-\frac12}]=2(a_{n-2}-a_{n-3})\sim\frac{2^\frac14}{\sqrt\pi n^\frac12}(3-2\sqrt2)^{-n+\frac32}(2-\sqrt2).
\end{equation}
From \eqref{pieceofHasymp} and
 \eqref{snasympappendix}, we obtain
\begin{equation}\label{pieceof2ndmomA--}
\frac{t^n[2t^2(1-t)(t^2-6t+1)^{-\frac12}]}{s_n}\sim2(3-2\sqrt2)(2-\sqrt2)n=(20-14\sqrt2)n.
\end{equation}
From  \eqref{H--from+-}, \eqref{pieceof2ndmomA--} and
\eqref{finalform},
we obtain
\begin{equation}\label{final2ndmomA--}
\begin{aligned}
&\frac{t^n[H^{-,-}]}{s_n}=(6-4\sqrt2)n^2+(3-2\sqrt2)n+(20-14\sqrt2)n+o(n)=\\
&(6-4\sqrt2)n^2+(23-16\sqrt2)n+o(n).
\end{aligned}
\end{equation}
From \eqref{varianceagain}, \eqref{firstmomsqagain} and  \eqref{final2ndmomA--}, we conclude that
$$
\text{Var}^{\text{sep}}_n(A_n^{-,-})\sim(23-16\sqrt2)n-\frac{30-21\sqrt2}2n=\frac{16-11\sqrt2}2n,
$$
which proves \eqref{varasympagain}.
\hfill $\square$

\section{Appendix}\label{append}
\begin{proposition}\label{snasympprop}
\begin{equation}\label{snasympappendix}
s_n\sim\frac1{2^\frac34\sqrt\pi}(3-2\sqrt2)^{-n+\frac12}\thinspace n^{-\frac32}.
\end{equation}
With more precision,
\begin{equation}\label{snasymprefine}
s_n=\frac1{\sqrt\pi}(3-2\sqrt2)^{-n+\frac12}
\Big(\frac1{2^\frac34}n^{-\frac32}+\frac{12\sqrt2-9}{2^\frac14\cdot32}n^{-\frac52}+o(n^{-\frac52})\Big).
\end{equation}
\end{proposition}
\begin{proof}
Denote the two roots of $t^2-6t+1$ by
\begin{equation*}\label{rootsr1r2}
r_1=3-2\sqrt2,\ \ r_2=3+2\sqrt2.
\end{equation*}
Define the sequence $\{b_n\}_{n=0}^\infty$ by
\begin{equation}\label{bn}
\sqrt{t^2-6t+1}=\sum_{n=0}^\infty b_nt^n,\ |t|<3-2\sqrt2.
\end{equation}
Write
\begin{equation}\label{product}
\sqrt{t^2-6t+1}=\sqrt{1-\frac t{r_1}}\thinspace\sqrt{1-\frac t{r_2}}.
\end{equation}
The Taylor series for $\sqrt{1-x}$ is given by
\begin{equation}\label{sqrt1-x}
\sqrt{1-x}=1-\sum_{n=1}^\infty\frac1{2n-1}\frac{(2n)!}{(n!)^22^{2n}}x^n.
\end{equation}
From \eqref{product} and  \eqref{sqrt1-x}, we have
\begin{equation}\label{productagain}
\begin{aligned}
&\sqrt{t^2-6t+1}=1-\sum_{n=1}^\infty\frac1{2n-1}\frac{(2n)!}{(n!)^22^{2n}}\frac1{r_1^n}t^n-\sum_{n=1}^\infty\frac1{2n-1}\frac{(2n)!}{(n!)^22^{2n}}\frac1{r_2^n}t^n+\\
&\sum_{n=2}^\infty\Big(\sum_{j=1}^{n-1} \frac1{2(n-j)-1}\frac{(2n-2j)!}{((n-j)!)^22^{2(n-j)}}\frac1{r_1^{n-j}}\frac1{2j-1}\frac{(2j)!}{(j!)^22^{2j}}\frac1{r_2^j}\Big)t^n.
\end{aligned}
\end{equation}

We write the coefficient of $t^n$ in the second line of \eqref{productagain} as
\begin{equation}\label{2ndlinecoeff}
\begin{aligned}
&\sum_{j=1}^{n-1} \frac1{2(n-j)-1}\frac{(2n-2j)!}{((n-j)!)^22^{2(n-j)}}\frac1{r_1^{n-j}}\frac1{2j-1}\frac{(2j)!}{(j!)^22^{2j}}\frac1{r_2^j}=\\
&\frac1{r_1^n}\sum_{j=1}^{n-1} \frac1{2(n-j)-1}\frac{(2n-2j)!}{((n-j)!)^22^{2(n-j)}}\frac1{2j-1}\frac{(2j)!}{(j!)^22^{2j}}(\frac{r_1}{r_2})^j.
\end{aligned}
\end{equation}
From \eqref{sqrt1-x}  we have
\begin{equation}\label{convexpremainder}
\begin{aligned}
&\sum_{j=1}^\infty\frac1{2j-1}\frac{(2j)!}{(j!)^22^{2j}}(\frac{r_1}{r_2})^j=1-\sqrt{1-\frac{r_1}{r_2}}
=1-\frac{2^\frac54}{\sqrt{3+2\sqrt2}}=
1-2^\frac54\sqrt{3-2\sqrt2}.
\end{aligned}
\end{equation}
From Stirling's formula, $n!\sim n^ne^{-n}\sqrt{2\pi n}$, one obtains the well-known asymptotic formula
\begin{equation}\label{stirling}
\frac{(2n)!}{(n!)^22^{2n}}\sim\frac1{\sqrt{\pi n}}.
\end{equation}
By \eqref{stirling}, the expression   $\frac1{2j-1}\frac{(2j)!}{(j!)^22^{2j}}$ multiplying $(\frac{r_1}{r_2})^j$ in \eqref{convexpremainder} is bounded in $j$.
 Consequently, for some $C>0$,
 \begin{equation}\label{someC}
 \sum_{j=M}^\infty\frac1{2j-1}\frac{(2j)!}{(j!)^22^{2j}}(\frac{r_1}{r_2})^j\le C\sum_{j=M}^\infty (\frac{r_1}{r_2})^j,
\text{for any}\ M\in\mathbb{N}.
\end{equation}
 From this  we conclude that
 \begin{equation}\label{geomseries}
 \begin{aligned}
 &\text{for any}\ l\in\mathbb{N},\ \text{there exists a constant}\ C_l>0\ \text{such that
 if}\  M_n=[C_l\log n]\in\mathbb{N},\ \text{then}\\
 &\sum_{j=M_n+1}^\infty\frac1{2j-1}\frac{(2j)!}{(j!)^22^{2j}}(\frac{r_1}{r_2})^j\le n^{-l}.
\end{aligned}
\end{equation}
By \eqref{stirling},
\begin{equation}\label{n-jasymp}
\begin{aligned}
& \frac1{2(n-j)-1}\frac{(2n-2j)!}{((n-j)!)^22^{2(n-j)}}\sim\frac1{2\sqrt{\pi }}\frac1{n^\frac32},\ \text{uniformly over }\ j\in\{1,\cdots, [C_l\log n]\}\\
& \text{as}\ n\to\infty,
\end{aligned}
\end{equation}
From  \eqref{2ndlinecoeff}, \eqref{convexpremainder}, \eqref{geomseries} and \eqref{n-jasymp}, it follows that
the coefficient of $t^n$ in the second line of \eqref{productagain} satisfies
\begin{equation}\label{2ndlinecoeffasymp}
\begin{aligned}
&\sum_{j=1}^{n-1} \frac1{2(n-j)-1}\frac{(2n-2j)!}{((n-j)!)^22^{2(n-j)}}\frac1{r_1^{n-j}}\frac1{2j-1}\frac{(2j)!}{(j!)^22^{2j}}\frac1{r_2^j}\sim\\
&\big(1-2^\frac54\sqrt{3-2\sqrt2}\big)\frac1{2\sqrt\pi}\frac1{n^\frac32}\frac1{r_1^n}.
\end{aligned}
\end{equation}

From \eqref{stirling} or \eqref{n-jasymp}, the coefficient of $t^n$ in the first sum on the first line of \eqref{productagain} satisfies
\begin{equation}\label{firstlinefirstsum}
\frac1{2n-1}\frac{(2n)!}{(n!)^22^{2n}}\frac1{r_1^n}\sim\frac1{2\sqrt\pi}\frac1{n^\frac32}\frac1{r_1^n}.
\end{equation}
Since $r_2>r_1$, the coefficient of $t^n$ in the second sum on the first line of \eqref{productagain} is exponentially smaller than that of the first sum on the first line of \eqref{productagain}.
Using this last fact with \eqref{bn}, \eqref{productagain}, \eqref{2ndlinecoeffasymp} and \eqref{firstlinefirstsum}, it follows that
\begin{equation}\label{bnasymp}
b_n\sim-2^\frac14\frac1{\sqrt\pi}(3-2\sqrt2)^{-n+\frac12}\thinspace\frac1{n^\frac32}.
\end{equation}
Now  \eqref{snasympappendix} follows from \eqref{bnasymp}, \eqref{bn} and \eqref{sgenfunc}.

We now turn to \eqref{snasymprefine}.
A refined form of Stirling's formula gives
$n!=n^ne^{-n}\sqrt{2\pi n}\big(1+\frac1{12 n}+O(\frac1{n^2})\big)$  .
Thus, recalling \eqref{stirling}, we  obtain
\begin{equation}\label{stirlingrefine}
\frac{(2n)!}{(n!)^22^{2n}}=\frac1{\sqrt{\pi n}}\frac{1+\frac1{24n}+O(\frac1{n^2})}{\big(1+\frac1{12n}+O(\frac1{n^2})\big)^2}=\frac1{\sqrt{\pi n}}\big(1-\frac1{8n}+O(\frac1{n^2})\big).
\end{equation}
Thus, the coefficient of $t^n$ in the first sum on the first line of \eqref{productagain} satisfies
\begin{equation}\label{firstlinefirstsum+}
\begin{aligned}
&\frac1{2n-1}\frac{(2n)!}{(n!)^22^{2n}}\frac1{r_1^n}=\frac{2n}{2n-1}\thinspace\frac1{2\sqrt\pi}\frac1{n^\frac32}\frac1{r_1^n}\big(1-\frac1{8n}+O(\frac1{n^2})\big)=\\
&\big(1+\frac1{2n}+O(\frac1{n^2})\big)\thinspace\frac1{2\sqrt\pi}\frac1{n^\frac32}\frac1{r_1^n}\big(1-\frac1{8n}+O(\frac1{n^2})\big)=\\
&\frac1{2\sqrt\pi}\frac1{n^\frac32}\frac1{r_1^n}\big(1+\frac3{8n}+O(\frac1{n^2})\big).
\end{aligned}
\end{equation}
As before, we don't need to consider the coefficient of $t^n$ in the second sum on the first line of \eqref{productagain} since it is exponentially smaller than that of the first sum on the right hand side
of \eqref{productagain}.

Now consider the coefficient of $t^n$ in the second line of
\eqref{productagain} as given by the right hand side of \eqref{2ndlinecoeff}.
We write this as
\begin{equation}\label{2ndlinecoeff-split}
\begin{aligned}
&\frac1{r_1^n}\sum_{j=1}^{n-1} \frac1{2(n-j)-1}\frac{(2n-2j)!}{((n-j)!)^22^{2(n-j)}}\frac1{2j-1}\frac{(2j)!}{(j!)^22^{2j}}(\frac{r_1}{r_2})^j=\\
&\frac1{r_1^n} \frac1{2n-1}\frac{(2n)!}{(n!)^22^{2n}}\sum_{j=1}^{n-1}\frac1{2j-1}\frac{(2j)!}{(j!)^22^{2j}}(\frac{r_1}{r_2})^j+\\
&\frac1{r_1^n}\sum_{j=1}^{M_n}\Big( \frac1{2(n-j)-1}\frac{(2n-2j)!}{((n-j)!)^22^{2(n-j)}}-\frac1{2n-1}\frac{(2n)!}{(n!)^22^{2n}}\Big)\frac1{2j-1}\frac{(2j)!}{(j!)^22^{2j}}(\frac{r_1}{r_2})^j+\\
&\frac1{r_1^n}\sum_{j=M_n}^{n-1}\Big( \frac1{2(n-j)-1}\frac{(2n-2j)!}{((n-j)!)^22^{2(n-j)}}-\frac1{2n-1}\frac{(2n)!}{(n!)^22^{2n}}\Big)\frac1{2j-1}\frac{(2j)!}{(j!)^22^{2j}}(\frac{r_1}{r_2})^j:=\\
&\frac1{r_1^n}\big(I_n+II_n+III_n\big),
\end{aligned}
\end{equation}
where $M_n=[C_l\log n]$ with $C_l$ is as in \eqref{geomseries} and $l>\frac52$.
Using  \eqref{convexpremainder},  \eqref{someC}  and \eqref{firstlinefirstsum+},
 we conclude that
 \begin{equation}\label{I}
\begin{aligned}
& I_n=\big(1-2^\frac54\sqrt{3-2\sqrt2}+R_n\big)\frac1{2\sqrt\pi}\frac1{n^\frac32}\big(1+\frac3{8n}+O(\frac1{n^2})\big)=\\
&\big(1-2^\frac54\sqrt{3-2\sqrt2}\big)\frac1{2\sqrt\pi}\frac1{n^\frac32}\big(1+\frac3{8n}+O(\frac1{n^2})\big),
\end{aligned}
 \end{equation}
where $R_n$ decays exponentially.

By the same reasoning that led to \eqref{someC},
 for some $C>0$,
 $|III_n|\le C\sum_{j=M_n}^\infty (\frac{r_1}{r_2})^j$. Thus,
 \begin{equation}\label{III}
 \begin{aligned}
 &\text{for any}\ l\in\mathbb{N},\ \text{there exists a constant}\ C_l>0\ \text{such that
 if we choose}\\
 & M_n=[C_l\log n],\ \text{then}\ |III_n|\le n^{-l}.
\end{aligned}
\end{equation}
(Without loss of generality, we can choose the same constant $C_l$ in   \eqref{geomseries} and \eqref{III}.)
As above, we choose $l>\frac52$.

Consider now $II_n$. We write
\begin{equation}\label{forII}
\begin{aligned}
 &\frac1{2(n-j)-1}\frac{(2n-2j)!}{((n-j)!)^22^{2(n-j)}}-\frac1{2n-1}\frac{(2n)!}{(n!)^22^{2n}}=\frac1{2n-1}\frac{(2n)!}{(n!)^22^{2n}}\big(\frac{B_{n-j}}{B_n}-1\big),
\end{aligned}
\end{equation}
where
$B_k=\frac1{2k-1}\frac{(2k)!}{(k!)^22^{2k}}$.
By \eqref{firstlinefirstsum+},
\begin{equation}\label{forIIagain}
\begin{aligned}
&\frac{B_{n-j}}{B_n}=\frac{\frac1{2\sqrt\pi}\frac1{(n-j)^\frac32}\big(1+\frac3{8(n-j)}+O(\frac1{n^2})\big)}
{\frac1{2\sqrt\pi}\frac1{n^\frac32}\big(1+\frac3{8n}+O(\frac1{n^2})\big)}=
(\frac n{n-j})^\frac32\big(1+O(\frac{\log n}{n^2})\big)=\\
&\big(1+\frac32\frac j{n-j}+O(\frac1{n^2})\big)\big(1+O(\frac{\log n}{n^2})\big)= 1+\frac32\frac jn+O(\frac{\log n}{n^2}),\\
&\text{uniformly over}\ j\in\{1,\cdots, C_l\log n\}.
\end{aligned}
\end{equation}
Substituting \eqref{forII} and \eqref{forIIagain} in the formula for $II_n$ in \eqref{2ndlinecoeff-split}, we have
\begin{equation}\label{IInest}
II_n= \frac1{2n-1}\frac{(2n)!}{(n!)^22^{2n}}
\sum_{j=1}^{C_l\log n}\big(\frac32\frac jn+O(\frac{\log n}{n^2})\big) \frac1{2j-1}\frac{(2j)!}{(j!)^22^{2j}}(\frac{r_1}{r_2})^j.
\end{equation}
We write
\begin{equation}\label{intermediate}
\begin{aligned}
&\sum_{j=1}^\infty\big(\frac{3j}2\frac1{2j-1}\frac{(2j)!}{(j!)^22^{2j}}(\frac{r_1}{r_2})^j=
\big[\frac32x\big(\sum_{j=1}^\infty\frac1{2j-1}\frac{(2j)!}{(j!)^22^{2j}}x^j\big)\big]'|_{x=\frac{r_1}{r_2}}=\\
&\frac32x\big(1-(1-x)^\frac12\big)'|_{x=\frac{r_1}{r_2}}=\frac34\frac{r_1}{r_2}(1-\frac{r_1}{r_2})^{-\frac12}=\frac38\frac{r_1^\frac32}{2^\frac14},
\end{aligned}
\end{equation}
where we have used the fact that $r_2=\frac1{r_1}$.
 Using \eqref{geomseries}  with
\eqref{intermediate},
we have
\begin{equation}\label{keyIIn}
\sum_{j=1}^{C_l\log n}\big(\frac32\frac jn\frac1{2j-1}\frac{(2j)!}{(j!)^22^{2j}}(\frac{r_1}{r_2})^j=
\frac1n\big(\frac38\frac{r_1^\frac32}{2^\frac14}+O(n^{-l})\big).
\end{equation}
From \eqref{firstlinefirstsum+}, \eqref{IInest} and \eqref{keyIIn}, we conclude that
\begin{equation}\label{II}
\begin{aligned}
&II_n=\frac1{2\sqrt\pi}\frac1{n^\frac32}\Big(1+\frac3{8n}+O(\frac1{n^2})\Big)\big(\frac38\frac{r_1^\frac32}{2^\frac14}\frac1n+O(\frac{\log n}{n^2})\big)=\\
&\frac3{16\sqrt\pi 2^\frac14}r_1^\frac32\frac1{n^\frac52}+O(\frac1{n^\frac72}).
\end{aligned}
\end{equation}
Now from \eqref{productagain}, \eqref{firstlinefirstsum+}-\eqref{III} and \eqref{II}, we have
\begin{equation}\label{bnrefine}
\begin{aligned}
&b_n=\frac{(3-2\sqrt2)^{-n+\frac12}}{\sqrt\pi}\Big(-\frac{2^\frac14}{n^\frac32}-\frac38\frac{2^{\frac14}}{n^\frac52}+
\frac3{16}\frac{3-2\sqrt2}{2^\frac14}\frac1{n^\frac52}+
o(\frac1{n^\frac52})\Big)=\\
&\frac{(3-2\sqrt2)^{-n+\frac12}}{\sqrt\pi}\Big(-\frac{2^\frac14}{n^\frac32}+\frac{9-12\sqrt2}{16\cdot2^\frac14}\frac1{n^\frac52}+o(\frac1{n^\frac52})\Big).
\end{aligned}
\end{equation}
Now \eqref{snasymprefine} follows from \eqref{bnrefine} and \eqref{sgenfunc}.
\end{proof}

\begin{proposition}\label{oneover}
Let $\{a_n\}_{n=0}^\infty$ denote the coefficients of the power series about zero of
$\frac1{\sqrt{t^2-6t+1}}$\text{\rm:} $\frac1{\sqrt{t^2-6t+1}}=\sum_{n=0}^\infty a_nt^n$. Then
\begin{equation}\label{oneoverasymp}
a_n\sim\frac1{2^\frac54\sqrt\pi n^\frac12}(3-2\sqrt2)^{-n-\frac12}.
\end{equation}
With more precision,
\begin{equation}\label{oneoverasymprefine}
a_n=\frac{(3-2\sqrt2)^{-n-\frac12}}{\sqrt\pi}\Big(\frac1{2^\frac54}\frac1{n^\frac12}+
\frac{3-4\sqrt2}{32\cdot2^\frac34}\frac1{n^\frac32}+o(\frac1{n^\frac32})\Big).
\end{equation}
\end{proposition}
\begin{proof}
We will show that \eqref{oneoverasymprefine} follows readily from \eqref{-32asymp} in Proposition \ref{-32} below.
Of course \eqref{oneoverasymp} is contained in \eqref{oneoverasymprefine}.
However, we note that the proof of \eqref{oneoverasymp} is considerably shorter than the proof of \eqref{oneoverasymprefine},
just as the proof of \eqref{snasympappendix} was considerably shorter than that of \eqref{snasymprefine}.
As was mentioned at the end of section \ref{intro}, for the proof of Theorem \ref{mean} the only results we use from the appendix are \eqref{snasympappendix} and \eqref{oneoverasymp}.

Differentiating gives
$$
\sum_{n=1}^\infty na_nt^{n-1}=(3-t)(t^2-6t+1)^{-\frac32}=(3-t)\sum_{n=0}^\infty\alpha_n t^n,
$$
where $\alpha_n$ is as in Proposition \ref{-32}. From this we obtain
\begin{equation}\label{aalpha}
a_n=\frac3n\alpha_{n-1}-\frac1n\alpha_{n-2}.
\end{equation}

Thus, from \eqref{-32asymp}, we have
\begin{equation}\label{aasymp}
\begin{aligned}
&\alpha_n=\frac{(3-2\sqrt2)^{-n-\frac12}}{\sqrt\pi}\Big(
\frac3{4\cdot2^\frac34}\frac{(n-1)^\frac12}n+
\frac{3(24-9\sqrt2)}{128\cdot2^\frac34}\frac1{(n-1)^\frac12n}\Big)-\\
&\frac{(3-2\sqrt2)^{-n-\frac12}}{\sqrt\pi}\Big(\frac{3-2\sqrt2}{4\cdot2^\frac34}\frac{(n-2)^\frac12}n+
\frac{(3-2\sqrt2)(24-9\sqrt2)}{128\cdot2^\frac34} \frac1{(n-2)^\frac12n}\Big)+\\
&o\Big((3-2\sqrt2)^{-n}\frac1{n^\frac32}\Big).
\end{aligned}
\end{equation}
Writing
$$
\frac{(n-i)^\frac12}n=\frac{n^\frac12(1-\frac in)^\frac12}n=n^{-\frac12}-\frac i2n^{-\frac32}+o(n^{-\frac32}),\ i=1,2,
$$
and
$$
\frac1{(n-i)^\frac12n}=\frac1{n^\frac32}+o(\frac1{n^\frac32}),\ i=1,2,
$$
and substituting in \eqref{aasymp},
we obtain
\begin{equation}
\begin{aligned}
&a_n=\frac{(3-2\sqrt2)^{-n-\frac12}}{\sqrt\pi}\big(\frac3{4\cdot2^\frac34}-\frac{3-2\sqrt2}{4\cdot2^\frac34}\big)n^{-\frac12}+\\
&\frac{(3-2\sqrt2)^{-n-\frac12}}{\sqrt\pi}\big(-\frac12\frac3{4\cdot2^\frac34}+\frac{3(24-9\sqrt2)}{128\cdot2^\frac34}+
\frac{3-2\sqrt2}{4\cdot2^\frac34}-\frac{(3-2\sqrt2)(24-9\sqrt2)}{128\cdot2^\frac34}\big)n^{-\frac32}+\\
&o\Big((3-2\sqrt2)^{-n}\frac1{n^\frac32}\Big)=\frac{(3-2\sqrt2)^{-n-\frac12}}{\sqrt\pi}\Big(\frac1{2^\frac54}\frac1{n^\frac12}+
\frac{3-4\sqrt2}{32\cdot2^\frac34}\frac1{n^\frac32}+o(\frac1{n^\frac32})\Big),
\end{aligned}
\end{equation}
which is \eqref{oneoverasymprefine}.

\end{proof}

\begin{proposition}\label{-32}
Let $\{\alpha_n\}_{n=0}^\infty$ denote the coefficients of the power series about zero of
$\frac1{(t^2-6t+1)^\frac32}$\text{\rm:} $\frac1{(t^2-6t+1)^\frac32}=\sum_{n=0}^\infty \alpha_nt^n$. Then
\begin{equation}\label{-32asymp}
\alpha_n=\frac{(3-2\sqrt2)^{-n-\frac32}}{\sqrt\pi}\Big(\frac1{4\cdot2^\frac34}n^\frac12+\frac{24-9\sqrt2}{128\cdot2^\frac34}\frac1{n^\frac12}+o(\frac1{n^\frac12})\Big).
\end{equation}
\end{proposition}
\begin{proof}
The proof is similar to the proof of  \eqref{snasymprefine}  in Proposition \ref{snasympprop}.
Differentiating
\eqref{sqrt1-x} twice, one obtains
\begin{equation}\label{1-xto-3/2}
(1-x)^{-\frac32}=\sum_{n=0}^\infty\frac{(2n+1)!}{(n!)^22^{2n}}x^n.
\end{equation}
Let $r_1,r_2$ be the roots of $t^2-6t+1$ as  in the proof of Proposition \ref{snasympprop}.
By \eqref{product} and \eqref{1-xto-3/2},
we have
\begin{equation}\label{3/2expand}
\begin{aligned}
&\frac1{(t^2-6t+1)^\frac32}=1+\sum_{n=1}^\infty\frac{(2n+1)!}{(n!)^22^{2n}}\frac1{r_1^n}t^n+
\sum_{n=1}^\infty\frac{(2n+1)!}{(n!)^22^{2n}}\frac1{r_2^n}t^n+\\
&\sum_{n=2}^\infty\Big(\sum_{j=1}^{n-1}\frac{(2n-2j+1)!}{((n-j)!)^22^{2(n-j)}}\frac1{r_1^{n-j}}\frac{(2j+1)!}{(j!)^22^{2j}}\frac1{r_2^j} \Big)t^n.
\end{aligned}
\end{equation}
The coefficient of $t^n$ in the second term on the right hand side of \eqref{3/2expand} is exponentially smaller that of  the
first term, so we can ignore it. Using \eqref{stirlingrefine}, we have
\begin{equation}\label{stirlingfor32}
\begin{aligned}
&\frac{(2n+1)!}{(n!)^22^{2n}}=(2n+1)\frac1{\sqrt{\pi n}}\big(1-\frac1{8n}+O(\frac1{n^2})\big)=(1+\frac1{2n})\frac{2n^\frac12}{\sqrt{\pi}}\big(1-\frac1{8n}+O(\frac1{n^2})\big)=\\
&\frac{2n^\frac12}{\sqrt{\pi}}\big(1+\frac3{8n}+O(\frac1{n^2})\big).
\end{aligned}
\end{equation}
So the coefficient of $t^n$ in the first term on the right hand side of \eqref{3/2expand} is given by
\begin{equation}\label{coefffirstterm32}
\frac{2n^\frac12}{\sqrt{\pi}r_1^n}\big(1+\frac3{8n}+O(\frac1{n^2})\big).
\end{equation}
We split the coefficient of $t^n$ in the second line of \eqref{3/2expand} into three parts, $I_n,II_n,III_n$, just as we did for the coefficient of $t^n$ in the second line of
\eqref{productagain} in \eqref{2ndlinecoeff-split}. Thus, we write this term as
\begin{equation}\label{2ndlinecoeff-splitagainagain}
\begin{aligned}
&\sum_{j=1}^{n-1}\frac{(2n-2j+1)!}{((n-j)!)^22^{2(n-j)}}\frac1{r_1^{n-j}}\frac{(2j+1)!}{(j!)^22^{2j}}\frac1{r_2^j}=\\
&\frac1{r_1^n}\sum_{j=1}^{n-1} \frac{(2n-2j+1)!}{((n-j)!)^22^{2(n-j)}}\frac{(2j+1)!}{(j!)^22^{2j}}(\frac{r_1}{r_2})^j=\\
&\frac1{r_1^n} \frac{(2n+1)!}{(n!)^22^{2n}}\sum_{j=1}^{n-1}\frac{(2j+1)!}{(j!)^22^{2j}}(\frac{r_1}{r_2})^j+\\
&\frac1{r_1^n}\sum_{j=1}^{M_n}\Big(\frac{(2n-2j+1)!}{((n-j)!)^22^{2(n-j)}}-\frac{(2n+1)!}{(n!)^22^{2n}}\Big)\frac{(2j+1)!}{(j!)^22^{2j}}(\frac{r_1}{r_2})^j+\\
&\frac1{r_1^n}\sum_{j=M_n}^{n-1}\Big(\frac{(2n-2j+1)!}{((n-j)!)^22^{2(n-j)}}-\frac{(2n+1)!}{(n!)^22^{2n}}\Big)\frac{(2j+1)!}{(j!)^22^{2j}}(\frac{r_1}{r_2})^j:=\\
&\frac1{r_1^n}\big(I_n+II_n+III_n\big),
\end{aligned}
\end{equation}
where $M_n=[C_l\log n]$, with $C_l$ sufficiently large so that
\begin{equation}\label{IIIagainagain}
 |III_n|\le n^{-l},
\end{equation}
similar to \eqref{III}.
This time, we choose $l>\frac12$.

Similar to the derivation of \eqref{I}, we use
 \eqref{1-xto-3/2} and \eqref{stirlingfor32} to obtain
\begin{equation}\label{Iagainagain}
\begin{aligned}
&I_n=\frac{2n^\frac12}{\sqrt{\pi}}\big(1+\frac3{8n}+O(\frac1{n^2})\big)\big((1-\frac{r_1}{r_2})^{-\frac32}-1+R_n\big)=\\
&\frac{2n^\frac12}{\sqrt{\pi}}\big(1+\frac3{8n}+O(\frac1{n^2})\big)\big(\frac1{8(3-2\sqrt2)^\frac322^\frac34}-1\big),
\end{aligned}
\end{equation}
where $R_n$ decays exponentially.

We now turn to $II_n$. In the present case we have
$$
II_n=\sum_{j=1}^{M_n}\Big(\frac{(2n-2j+1)!}{((n-j)!)^22^{2(n-j)}}-\frac{(2n+1)!}{(n!)^22^{2n}}\Big)\frac{(2j+1)!}{(j!)^22^{2j}}(\frac{r_1}{r_2})^j.
$$
Similar to \eqref{forII} and \eqref{forIIagain}, we have
\begin{equation*}\label{forII-prop3}
\begin{aligned}
 &\frac{(2n-2j+1)!}{((n-j)!)^22^{2(n-j)}}-\frac{(2n+1)!}{(n!)^22^{2n}}=\frac{(2n+1)!}{(n!)^22^{2n}}\big(\frac{B_{n-j}}{B_n}-1\big),
\end{aligned}
\end{equation*}
where now
$B_k=\frac{(2k+1)!}{(k!)^22^{2k}}$.
By  \eqref{stirlingfor32},
\begin{equation*}\label{Bkasymp32}
\begin{aligned}
&\frac{B_{n-j}}{B_n}=\frac{\frac{2(n-j)^\frac12}{\sqrt{\pi}}\big(1+\frac3{8(n-j)}+O(\frac1{n^2})\big)}{\frac{2n^\frac12}{\sqrt{\pi}}\big(1+\frac3{8n}+O(\frac1{n^2})\big)}=
(\frac{n-j}n)^\frac12\big(1+O(\frac{\log n}{n^2})\big)=\\
&\big(1-\frac j{2n}+O(\frac1{n^2})\big)\big(1+O(\frac{\log n}{n^2})\big)=1-\frac j{2n}+O(\frac{\log n}{n^2}),\\
&\text{uniformly over}\ j\in\{1,\cdots, C_l\log n\}.
\end{aligned}
\end{equation*}
Thus, similar to \eqref{IInest}, we obtain
\begin{equation}\label{IInestagainagain}
II_n=\frac{(2n+1)!}{(n!)^22^{2n}}\big(\sum_{j=1}^{C_l\log n}\big(-\frac j{2n}+O(\frac{\log n}{n^2})\big)
\frac{(2j+1)!}{(j!)^22^{2j}}(\frac{r_1}{r_2})^j.
\end{equation}
We write
\begin{equation}\label{usederiv32}
\begin{aligned}
&\sum_{j=1}^\infty\frac j2\frac{(2j+1)!}{(j!)^22^{2j}}(\frac{r_1}{r_2})^j=\big[\frac x2\big(\sum_{j=1}^\infty\frac{(2j+1)!}{(j!)^22^{2j}}x^j\big)'\big]|_{x=\frac{r_1}{r_2}}=\\
&\big[\frac x2\big((1-x)^{-\frac32}-1\big)'\big]|_{x=\frac{r_1}{r_2}}=\frac34\frac{r_1}{r_2}(1-\frac{r_1}{r_2})^{-\frac52}=\frac3{256\cdot2^\frac14(3-2\sqrt2)^\frac12},
\end{aligned}
\end{equation}
where we have used the fact that $r_2=\frac1{r_1}$.
Using \eqref{usederiv32} and \eqref{stirlingfor32} in  \eqref{IInestagainagain}, we obtain, similar to \eqref{II},
\begin{equation}\label{IIagainagain}
\begin{aligned}
&II_n=\frac{2n^\frac12}{\sqrt{\pi}}\big(1+\frac3{8n}+O(\frac1{n^2})\big)\big(-\frac3{256\cdot2^\frac14(3-2\sqrt2)^\frac12}\frac1n+O(\frac{\log n}{n^2})\big)=\\
&-\frac1{\sqrt\pi}\frac3{128\cdot2^\frac14(3-2\sqrt2)^\frac12}\frac1{n^\frac12}+o(\frac1{n^\frac12}).
\end{aligned}
\end{equation}
From \eqref{3/2expand}, \eqref{coefffirstterm32}-\eqref{Iagainagain} and \eqref{IIagainagain}, we obtain
\begin{equation}\label{alphanfinal}
\begin{aligned}
&\alpha_n=\frac{2n^\frac12}{\sqrt{\pi}r_1^n}\big(1+\frac3{8n}+O(\frac1{n^2})\big)+
\frac{2n^\frac12}{\sqrt{\pi}r_1^n}\big(1+\frac3{8n}+O(\frac1{n^2})\big)\big(\frac1{8(3-2\sqrt2)^\frac322^\frac34}-1\big)-\\
&\frac1{\sqrt\pi}\frac3{128\cdot2^\frac14(3-2\sqrt2)^\frac12}\frac1{n^\frac12}=\\
&\frac{(3-2\sqrt2)^{-n-\frac32}}{\sqrt\pi}\Big(\frac1{4\cdot2^\frac34}n^\frac12+\frac3{32\cdot2^\frac34}\frac1{n^\frac12}-
\frac{3(3-2\sqrt2)}{128\cdot2^\frac14}\frac1{n^\frac12}+o(\frac1{n^\frac12})\Big)=\\
&\frac{(3-2\sqrt2)^{-n-\frac32}}{\sqrt\pi}\Big(\frac1{4\cdot2^\frac34}n^\frac12+\frac{24-9\sqrt2}{128\cdot2^\frac34}\frac1{n^\frac12}+o(\frac1{n^\frac12})\Big).
\end{aligned}
\end{equation}

\end{proof}

\end{document}